\documentclass[reqno]{amsart}

\usepackage{amscd,amssymb, mathrsfs}
\usepackage[cmtip,arrow,curve,matrix]{xy}

\usepackage{url}
\numberwithin{equation}{section}
\DeclareFontFamily{OMS}{rsfs}{\skewchar\font'60}
\DeclareFontShape{OMS}{rsfs}{m}{n}{<-5>rsfs5 <5-7>rsfs7 <7->rsfs10 }{}
\DeclareSymbolFont{rsfs}{OMS}{rsfs}{m}{n}
\DeclareSymbolFontAlphabet{\scr}{rsfs}

\newtheorem{Theorem}[equation]{Theorem}

\newtheorem{Lemma}[equation]{Lemma}
\newtheorem{Proposition}[equation]{Proposition}
\newtheorem{Corollary}[equation]{Corollary}

\theoremstyle{definition}

\newtheorem{Definition}[equation]{Definition}
\newtheorem{Remark}[equation]{Remark}

\newcommand{\alg}[2]{#2[#1]} \newcommand{\ainfty}{A_{\infty}}

\newcommand{\ainftyspaces}{\CatOf{$\ainfty$ spaces}}

\newcommand{\A}{\mathbb{A}}
\newcommand{\aA}{\mathcal{A}}

\DeclareMathOperator{\Aut}{Aut}
\DeclareMathOperator{\End}{End}

\newcommand{\CatOf}[1]{(\mbox{#1})}

\DeclareMathOperator*{\colim}{colim}

\newcommand{\eqdef}{\overset{\text{def}}{=}}
\newcommand{\einfty}{E_{\infty}}
\newcommand{\einftyspectra}{\spectra[\einfty]}

\newcommand{\EKMM}{\scr{M}}

\newcommand{\GLsym}{GL_{1}}
\newcommand{\GL}[1]{\GLsym #1}

\newcommand{\glsym}{gl_{1}}
\newcommand{\gl}[1]{\glsym #1}

\newcommand{\gplike}[2]{\alg{#1}{#2}^{\times}}

\DeclareMathOperator{\ho}{ho}

\newcommand{\heq}{\simeq}
\DeclareMathOperator{\Ho}{ho}
\renewcommand{\i}{\infty}

\renewcommand{\smash}{\wedge}
\newcommand{\Lsmash}{\smash^{\mathrm{L}}}

\newcommand{\iso}{\cong}

\newcommand{\linf}{\Omega^{\infty}}

\newcommand{\Lin}{\mathscr{L}}
\renewcommand{\L}{\mathbb{L}}

\newcommand{\monspectra}{\aS}

\newcommand{\plus}{+}
\newcommand{\pt}[1]{#1_{\plus}}
\newcommand{\ptC}{C_{*}}
\newcommand{\ptspace}{\ast}
\newcommand{\ptdspaces}{\spaces_{*}}

\newcommand{\Qeq}{\approx}
\newcommand{\hAut}{\textrm{hAut}}

\newcommand{\Rcf}{R^{\circ}}
\newcommand{\Rmod}{\Mod{R}}

\newcommand{\lRmodules}{\CatOf{right $\linf R$-modules}}
\newcommand{\Rmodules}{\CatOf{right $R$-modules}}

\newcommand{\spaces}{\TT}
\newcommand{\spectra}{\mathscr{S}} 
\newcommand{\aS}{\mathcal{S}}
\newcommand{\sinf}{\Sigma^{\infty}}

\newcommand{\splus}{\sinf_{+}} 
\newcommand{\Smash}{\wedge} 
\newcommand{\Svee}{\mathbb{S}}

\newcommand{\timesL}{\times_{\mathcal{L}}}
\newcommand{\T}{\mathbb{T}}

\newcommand{\universe}[1]{\mathscr{#1}}

\newcommand{\xra}[1]{\xrightarrow{#1}}

\newcommand{\Z}{\mathbb{Z}}

\let\catsymbfont\mathcal

\newcommand{\bL}{{\mathbb{L}}}
\newcommand{\bP}{{\mathbb{P}}}
\newcommand{\bT}{{\mathbb{T}}}

\newcommand{\M}{\mathscr{M}}
\newcommand{\C}{\mathscr{C}}

\newcommand{\TT}{\mathscr{T}}

\newcommand{\aC}{{\catsymbfont{C}}}
\newcommand{\aM}{\EKMM}

\newcommand{\sL}{{\mathcal{L}}}
\newcommand{\htp}{\simeq}

\DeclareMathOperator{\Map}{Map}

\newcommand{\Mod}[1]{{#1}\text{-}\mathrm{mod}}

\DeclareMathOperator{\SingTxt}{Sing}

\begin{document}

\title[Units, orientations, and Thom spectra]{Units of ring spectra, orientations, and Thom spectra via rigid infinite loop space theory}

\author[Ando]{Matthew Ando}
\address{Department of Mathematics \\
The University of Illinois at Urbana-Champaign \\
Urbana IL 61801 \\
USA} \email{mando@math.uiuc.edu}

\author[Blumberg]{Andrew J. Blumberg}
\address{Department of Mathematics \\University of Texas \\
Austin TX 78712}
\email{blumberg@math.utexas.edu}

\author[Gepner]{David Gepner}
\address{Department of Mathematics \\Purdue University \\ West Lafayette IN 47907 }
\email{dgepner@math.purdue.edu}

\author[Hopkins]{Michael J.~Hopkins}
\address{Department of Mathematics \\
Harvard University \\
Cambridge MA 02138}
\email{mjh@math.harvard.edu}

\author[Rezk]{Charles Rezk}

\address{Department of Mathematics \\
The University of Illinois at Urbana-Champaign \\
Urbana IL 61801 \\
USA}
\email{rezk@math.uiuc.edu}

\thanks{Ando was supported by NSF grant DMS-0705233 and DMS-1104746.
  Blumberg was partially supported by an NSF Postdoctoral Research
  Fellowship and by NSF grant DMS-0906105.  Gepner was supported by
  EPSRC grant EP/C52084X/1.  Hopkins was supported by the NSF.  Rezk
  was supported by NSF grant DMS-0505056 and DMS-1006054.}

\begin{abstract}
We extend the theory of Thom spectra and the associated obstruction
theory for orientations in order to support the construction of the
$\einfty$ string orientation of $tmf$, the spectrum of topological
modular forms. Specifically, we show that for an $\einfty$ ring
spectrum $A$, the classical construction of $\gl{A}$, the spectrum of
units, is the right adjoint of the functor
\[
\splus\linf\colon \Ho \CatOf{connective spectra} \rightarrow \Ho
\CatOf{$\einfty$ ring spectra}.
\]
To a map of spectra
\[
   f\colon b \to b\gl{A},
\]
we associate an $\einfty$ $A$-algebra Thom spectrum $Mf$, which admits
an $\einfty$
$A$-algebra map to $R$ if and only if the composition
\[
    b \to b\gl{A} \to b\gl{R}
\]
is null; the classical case developed by \cite{MQRT:ersers} arises
when $A$ is the sphere spectrum.  We develop the analogous theory
for $\ainfty$ ring spectra: if $A$ is an $\ainfty$ ring spectrum, then
to a map of spaces 
\[
     f\colon B \to B\GL{A}
\]
we associate an $A$-module Thom spectrum $Mf,$ which admits an
$R$-orientation if and only if
\[
B \to B\GL{A} \to B\GL{R}
\]
is null.  Our work is based on a new model of the Thom spectrum as a
derived smash product.
\end{abstract}

\maketitle

\tableofcontents

\section{Introduction}
\label{sec:prol-thom-isom}

In a forthcoming paper \cite{AHR:orientation}, three of us (Ando,
Hopkins, Rezk)  construct an $\einfty$ string orientation of $tmf$, the
spectrum of topological modular forms: more precisely, we construct a map
of $\einfty$ ring spectra from the Thom spectrum $MO\langle 8
\rangle$, also known as $MString$, to the spectrum $tmf$, whose value
on homotopy rings refines the Witten genus from $\pi_{*}MString$ to
the ring of integral modular forms for $SL_{2}\Z$.  
As explained by Hopkins in his ICM address \cite{Hopkins:icm2002}, the
argument requires a new formulation of the obstruction theory for
orientations of \cite{MQRT:ersers} in terms of the adjoint
relationship between the units of a commutative ring spectrum and
$\splus \linf$.   A central goal of this paper is to establish this
formulation.

This new picture of the obstruction theory is motivated by a description of the Thom spectrum originally due to the fourth author.  Another purpose of the paper is to study this construction of the Thom
spectrum.  For example, we use it to extend the classical theory by
developing an obstruction theory for orientations of $\ainfty$ ring
spectra.  We also use it to build Thom spectra in situations more
general than stable spherical fibrations; these more general
situations give rise to twisted generalized 
cohomology.  To carry out these extensions we use
certain relatively recently developed ``rigid'' point-set models for
$A_\infty$ (and $E_\infty$) spaces. 

\subsection{Recollection of the discrete case}
\label{sec:recoll-discr-case}

We begin by describing the algebraic model that motivates our approach.
Let $R$ be a discrete ring, and let $G=\GL{R}$.  A bundle of free rank-one
$R$-modules over $X$ is classified by a map $f\colon X \to BG$; let $P\to X$
be the associated principal $G$-bundle.   We'd like to attach an
$R$-module ``Thom spectrum'' $Mf$ to this situation, in such a way
that trivializations of $P$ over $X$ can be understood in terms of $R$-module
maps $Mf\to R$.  

For simplicity, we'll further assume that $X$ is
discrete.  Then $P$ is the $G$-set $P = \coprod_{x\in X} P_{x}$,
and we can form the $R$-module ``algebraic Thom spectrum''
\begin{equation} \label{eq:32}
  Mf =    \Z[P]\otimes_{\Z[G]} R.
\end{equation}
Formation of the tensor product uses the fact that the adjunction 
\[
\xymatrix{ {\Z\colon \CatOf{sets}} \ar@<.3ex>[r] & {\CatOf{abelian groups}}  \ar@<.3ex>[l] }
\]
induces an adjunction 
\[
\xymatrix{ {\Z\colon \CatOf{$G$-sets}} \ar@<.3ex>[r] & {\CatOf{$\Z[G]$-modules},}  \ar@<.3ex>[l] }
\]
so $\Z[P]$ is a $\Z[G]$-module.  Also, $\Z$ restricts to give an adjunction
\begin{equation} \label{eq:29}
\xymatrix{ {\Z\colon \CatOf{groups}} \ar@<.3ex>[r] & {\CatOf{rings}\colon
\GLsym,}  \ar@<.3ex>[l] }
\end{equation}
whose counit is the natural ring homomorphism 
\begin{equation}\label{eq:1}
\Z[G] \to R.
\end{equation}

Using these adjunctions, one checks easily that 
\[
  \CatOf{$R$--modules} (Mf,R) \iso \CatOf{$G$--sets} (P, R),
\]
and with respect to this isomorphism, the set of \emph{orientations} of
$Mf$ is the subset
\[
\xymatrix{ {\CatOf{$R$--modules} (Mf,R)} \ar@{-}[r]^-{\iso} &
{\CatOf{$G$--sets} (P, R)}
\\
{\CatOf{orientations} (Mf,R)} \ar@{-}[r]^-{\iso} \ar@{>->}[u] &
{\CatOf{$G$--sets} (P,G),} \ar@{>->}[u]}
\]
which in turn is isomorphic to the set of trivializations of the principal
$G$-bundle $P\to X$.

\subsection{The space of units and orientations}
\label{sec:space-units-orient}

Our approach to the Thom spectrum functor develops the approach sketched above for a general space $X$ and $\ainfty$ ring spectrum $R$.  Following \cite{MQRT:ersers}, 
when $R$ is an $A_\infty$ ring spectrum in the sense
of~\cite{LMS}, we can define the space of units of $R$ to be the
pullback in the diagram of (unpointed) spaces 
\[
\xymatrix{
\GL{R} \ar[r] \ar[d] & \linf{R} \ar[d] \\
(\pi_{0}R)^{\times} \ar[r] &\pi_{0}R.
}
\]
If $X$ is any space, then
\[
    [X,\GL{R}] = \{f\in R^{0} (\pt{X}) \,|\, \pi_{0}f (X) \subset
(\pi_{0}R)^{\times}\} = R^{0} (\pt{X})^{\times},
\]
which provides a justification for the definition.  More conceptually,
we show in Section~\ref{sec:space-units-derived} that this definition
of units can be interpreted as the space of automorphisms of $R$ (as an
$R$-module).

Working with the models of~\cite{LMS}, we have continuous (i.e.,
topologically enriched) adjunctions (analogous to \eqref{eq:29})
\begin{equation}\label{eq:15}
\xymatrix{{\CatOf{group-like $A_{\infty}$ spaces}} \ar@<.5ex>[r] &
{\ainftyspaces} \ar@<.5ex>[r]^-{\splus} \ar@<.5ex>[l]^-{\GLsym} &
{\CatOf{$\ainfty$ ring spectra}\colon \GLsym,} \ar@<.5ex>[l]^-{\linf} }
\end{equation}
where the right-hand adjunction is a special case of
\cite[p. 366]{LMS}.  Thus one can make sense of a map of $\ainfty$
ring spectra  $\splus \GL{R} \to R$ analogous to \eqref{eq:1}.

However, classical technology does not make it straightforward to describe the adjunction
\[
\xymatrix{ {\splus\colon \lRmodules} \ar@<.5ex>[r] & {\Rmodules\colon \linf}
\ar@<.5ex>[l] }
\]
and moreover, since $\GL{R}$ is not a topological group or monoid but
rather only a group-like $\ainfty$ space, it is not immediately
apparent how to form the (quasi)fibration
\[
    \GL{R} \to E\GL{R} \to B\GL{R},
\]
and then make sense of the construction \eqref{eq:32}.

Our strategy, which we carry out in \S\ref{sec:ainfty-thom-spectrum}, is
to use a ``rigid'' model of $A_\infty$ spaces.  Specifically, we use a
model of spaces equipped with a symmetric monoidal product such that
strict monoids for this product are precisely
$\ainfty$-spaces~\cite{Blumberg-Cohen-Schlichtkrull}.

In this setting, we can form a version of $\GL{R}$ which is a
group-like monoid, and then model $E\GL{R} \to B\GL{R}$ as a
quasi-fibration with an action of $\GL{R}$.  Given a map
\[
f\colon B\to B\GL{R},
\]
$\GL{R}$ acts on the pullback $P$ in the diagram
\[
\xymatrix{
P \ar[r] \ar[d] & E\GL{R} \ar[d] \\
B \ar[r]^-{f} & B\GL{R},
}
\]
and the spectrum $\splus P$ becomes a right $\splus
\GL{R}$-module.  We can then imitate \eqref{eq:32} to form an
$R$-module Thom spectrum as the derived smash product
\[
   Mf \eqdef \splus P \Lsmash_{\splus \GL{R}} R.
\]

With this definition, we find that 
\begin{equation} \label{eq:35}
   \CatOf{right $R$--modules} (Mf, R) \heq \CatOf{right
$\GL{R}$--spaces} (P,\linf R),
\end{equation}
where here (and in the remainder of this subsection) we are referring
to derived mapping spaces.

The space of \emph{orientations} of $Mf$ is the subspace of $R$-module 
maps $Mf \to R$ which correspond to
\[
   \CatOf{right $\GL{R}$--modules} (P,\GL{R}) \subset \CatOf{right
$\GL{R}$--modules} (P,\linf R).
\]
under the weak equivalence \eqref{eq:35}.  That is, we have a homotopy
pullback diagram
\[
\xymatrix{ {\CatOf{orientations} (Mf,R)} \ar[r]^-{\heq} \ar[d] &
{\CatOf{right $\GL{R}$--spaces} (P,\GL{R})} \ar[d]
\\
{\CatOf{right $R$--modules} (Mf, R)} \ar[r]^-{\heq} & {\CatOf{right
$\GL{R}$--modules} (P,\linf R).}  }
\]

We obtain an obstruction-theoretic characterization of the space of
orientations $Mf\to R$ as follows: it is weakly equivalent to the
derived space of lifts in the diagram 
\[
\xymatrix{ {P} \ar[r] \ar[d] & {E\GL{R}} \ar[d]
\\
{B} \ar[r]_-{f} \ar@{-->}[ur] & {B\GL{R}.}  }
\]
We are able to use this to recover the classical picture of an
orientation and also the Thom isomorphism.

Recall that a stable spherical fibration is classified by a map $B \to
BF$, where $F = \colim_V \hAut(S^V)$ (and the colimit is over
finite-dimensional subspaces of $\mathbb{R}^{\infty}$ and inclusions).
The space $BF$ gives a particularly convenient model for $B\GL{S}$.
The generalized construction we study in this paper associates an
$R$-module Thom spectrum $Mf$ to a map $f\colon B \to B\GL{R}$ for any
ring spectrum $R$; $f$ need not classify a stable spherical fibration.

To compare to the classical situation, we suppose that $f$ does arise
from a stable spherical fibration as the composite 
\[
     f\colon B \xra{g} B\GL{S} \xra{B\GL{\iota}} B\GL{R}.
\]
It follows directly from the definition that $Mf \heq Mg \Lsmash R$.

We define an $R$-orientation of $Mg$ to be a map of spectra $Mg \to R$ 
such that the induced map of $R$-modules $Mf \to R$ 
is an orientation as above.  We then can show that the space of
$R$-orientations of $Mg$ is the space of indicated lifts in the diagram
\[
\xymatrix{ {P} \ar[r] \ar[d] & {B (S,R)} \ar[r] \ar[d] & {E\GL{R}}
\ar[d]
\\
{B} \ar[r] \ar@{-->}[ur] & {B\GL{S}} \ar[r] & {B\GL{R}}, }
\]
where $B (S,R)$ is the pullback in the solid diagram.  This
generalizes to the $A_{\infty}$ case the work of May, Quinn, Ray, and 
Tornehave \cite{MR2162361,MQRT:ersers}. 

\begin{Remark}
In the companion paper \cite{ABGHR} we prove that when $g$
classifies a stable spherical fibration, then the spectrum $Mg$
constructed in this paper coincides with the Thom spectrum associated
to $g$ via the theory of~\cite{LMS}.
\end{Remark}

\subsection{The spectrum of units and $\einfty$ orientations}
\label{sec:obstr-theory-einfty}

To see how our constructions work when $R$ is an $\einfty$ ring
spectrum, once again it is illuminating first to consider the discrete
case.  Suppose that $R$ is a commutative ring.  Then $G=\GL{R}$ is an
abelian group, and we can choose a model of $BG$ that is an abelian
group as well.

Now suppose that $X$ is a discrete abelian group, and $f\colon X\to BG$ is
a homomorphism.  Then in the pullback diagram 
\[
\xymatrix{
\ar[d] P \ar[r] & \ar[d] EG \\
X \ar[r]^{f} & BG, \\
}
\]
$P \cong G \times X$ is an abelian group, and so the discrete ``Thom
spectrum''
\[
    Mf = \Z[P]\otimes_{\Z[G]} R \cong R[X]
\]
is a commutative ring: indeed it is the pushout in the diagram of
commutative rings
\[
\xymatrix{
   \Z[G] \ar[d] \ar[r] & R \ar[d] \\
   \Z[P] \ar[r] & Mf,
}
\]
where the homomorphism $\Z[G]\to R$ is the counit of the adjunction 
\[
\xymatrix{ {\Z \colon \CatOf{abelian groups}} \ar@<.3ex>[r] &
{\CatOf{commutative rings} \colon \GLsym,} \ar@<.3ex>[l] }
\]
which is the restriction to abelian groups of the adjunction
\eqref{eq:29}.

Turning to spaces and spectra, the adjunction \eqref{eq:15} restricts
to an adjunction 
\[
\xymatrix{{\CatOf{group-like $E_{\infty}$ spaces}} \ar@<.5ex>[r] &
{\CatOf{$\einfty$ spaces}} \ar@<.5ex>[r]^-{\splus} \ar@<.5ex>[l]^-{\GLsym} &
{\CatOf{$\einfty$ ring spectra}\colon \GLsym.} \ar@<.5ex>[l]^-{\linf} }
\]
In the $\einfty$ case there is the additional classical fact (e.g.,
see~\cite{MR0339152}) that the category of group-like $\einfty$ spaces
is a model for connective spectra: therefore if $R$ is an $\einfty$
ring spectrum then there is a spectrum $\gl{R}$ such that $\GL{R}\heq
\linf \gl{R}.$  Putting all this together, we see that the functor
$\glsym$ participates as the right adjoint in an adjunction 
\begin{equation}  \label{eq:33}
\xymatrix{ {\splus \linf\colon \Ho \CatOf{$(-1)$-connected spectra}}
\ar@<.3ex>[r] & {\Ho \CatOf{$\einfty$ ring spectra}\colon \glsym} \ar@<.3ex>[l] }
\end{equation}
which preserves the homotopy types of derived mapping spaces.

In contrast to the $A_\infty$ setting, this adjunction can be
constructed by assembling results in the literature, particularly work
of May.  However, as we worked through this we found it very useful to
reformulate the statements and proofs in a way which reflects advances
in the state of the art since the original work was done.  In 
Section~\ref{sec:constr-may-quinn}, we give a modern proof of this
adjunction, carefully rederiving and explaining the many classical
results involved. 

Assuming this development, in Section~\ref{sec:thom-spectra} we work
out the theory of $\einfty$ Thom spectra generalizing our new model of
$A_\infty$ Thom spectra and establish results about orientations as used in the construction of the String orientation of $tmf$.  

Let $R$ be an $\einfty$ ring spectrum, and suppose that $b$ is a
spectrum over $b\gl{R} = \Sigma\gl{R}$.  Let $p$ be the homotopy
pullback  
\begin{equation} \label{eq:13}
\xymatrix{ {\gl{R}} \ar@{=}[r] \ar[d] & {\gl{R}} \ar[d]
 \\
{p} \ar[r] \ar[d] & {e\gl{R}\heq \ptspace } \ar[d]
 \\
{b} \ar[r]^-{f} & {b\gl{R}.}  }
\end{equation}

The $E_\infty$ $R$-algebra Thom spectrum $Mf$ of $f\colon b\to
b\gl{R}$ is then defined to be the homotopy pushout in the diagram of
$\einfty$ $R$-algebras
\begin{equation}   \label{eq:6}
\xymatrix{
R \smash \splus \linf  \gl{R} \ar[r] \ar[d] & R \ar[d] \\
R \smash \splus \linf p \ar[r] & Mf,
}
\end{equation}
where the top map is induced from the counit of the adjunction
\eqref{eq:33}.  Since the homotopy pushout of $E_\infty$ ring spectra
coincides with the derived smash product, this generalizes the definition
in the $A_\infty$ setting.

For the obstruction theory, suppose $\varphi\colon R\to A$ is a map of
$\einfty$ ring spectra.  Then we have the solid commutative diagram 
\begin{equation} \label{eq:13A}
\xymatrix{ {\gl{R}} \ar[r] \ar[d] & {\gl{A}} \ar[d]
 \\
{p} \ar[r] \ar[d] \ar@{-->}[ur] & {e\gl{A}\heq \ptspace } \ar[d]
 \\
{b} \ar[r]^-{\tilde{\varphi}\circ f} \ar@{-->}[ur] & {b\gl{A},} }
\end{equation}
where we write $\tilde{\varphi}\colon bgl_1 R\to bgl_1 A$ for the induced
map.

Using the adjunction~\eqref{eq:33}, we prove that there is a homotopy pullback
diagram of derived mapping spaces (where $\spectra$ denotes the
category of spectra)  
\[
\xymatrix{
\CatOf{$\einfty$ $R$-$\mathrm{algebras}$} (Mf,A) \ar[r] \ar[d] &
\ar[d] \Map_{\spectra} (p,\gl{A})\\ 
\{\varphi \} \ar[r] & \Map_{\spectra} (\gl{R},\gl{A}).
}
\]
That is, the space of $R$-algebra maps $M\to A$ is weakly equivalent
to the space of lifts in the diagram \eqref{eq:13A}.

\subsection{Twisted generalized cohomology}

Our $R$-module Thom spectra locate ``twisted generalized cohomology''
in stable homotopy theory; from this point of view $B\GL{R}$
classifies the twists.  Let
\[
    f\colon X \to B\GL{R}
\]
be a map, and let $Mf$ be the asociated $R$-module Thom spectrum.  The
\emph{$f$-twisted $R$-homology} of 
$X$ is 
\[
     R^{f}_{k}X \eqdef \pi_{0}\Map_{\Rmod} (\Sigma^{k}R,Mf) \iso \pi_{k}Mf,
\]
while the \emph{$f$-twisted $R$-cohomology} of $X$ is
\[
    R_{f}^{k}X \eqdef \pi_{0}\Map_{\Rmod} (Mf, \Sigma^{k} R).
\]
If $f$ factors as
\begin{equation}\label{eq:17}
   f\colon X \xra{g} B\GL{S} \xra{i} B\GL{R},
\end{equation}
then we have $Mf \heq (Mg)\Lsmash R$ and so
\begin{align*}
  R_{k}^{f} (X) & = \pi_{k}Mf  \iso \pi_{k} ( Mg\Lsmash R)  = R_{k} Mg \\
  R^{k}_{f} (X) & = \pi_{0}\Map_{\Rmod} (Mf,\Sigma^{k}R) \iso \pi_{0}
  \Map_{\Mod{S}} (Mg,\Sigma^{k}R) \iso R^{k} Mg.
\end{align*}
That is, the $f$-twisted homology and cohomology coincide with the untwisted
$R$-homology and cohomology of the usual Thom spectrum of the
spherical fibration classified by $g$.  Thus the constructions in this
paper exhibit twisted generalized cohomology as the cohomology of a
generalized Thom spectrum.  In general the twists correspond to maps
$X \to B\GL{R};$ the ones which arise from Thom spectra of spherical
fibrations are the ones which factor as in \eqref{eq:17}.  We discuss
the relationship to other approaches to twisted generalized cohomology
in \cite{ABG}.

\subsection{Historical remarks and related work}

In his 1970 MIT notes \cite{MR2162361},\footnote{In the version available at
\url{http://www.maths.ed.ac.uk/~aar/books/gtop.pdf}, see the note on
page 236.} Sullivan introduced the classical obstruction theory for
orientations and suggested that Dold's theory of homotopy functors
\cite{MR0198464} could be used to construct the space $B (S,R)$ of
$R$-oriented spherical fibrations.  He also mentioned that the
technology to construct the delooping $B\GL{R}$ was on its way.  Soon
thereafter, May, Quinn, Ray, and Tornehave in \cite{MQRT:ersers}
constructed the space $B\GL{R}$ in the case that $R$ is an $\einfty$
ring spectrum, and described the associated obstruction theory for
orientations of spherical fibrations.  

Various aspects of the theory of units and Thom spectra have been
revisited by a number of authors as the foundations of stable homotopy
theory have advanced. For example, Schlichtkrull \cite{MR2057776}
studied the units of a symmetric ring spectrum, and May and Sigurdsson
\cite{MR2271789} have studied units and orientations from the perspective of their
categories of parametrized spectra.  Recently May has prepared an
authoritative paper revisiting operad (ring) spaces and operad (ring)
spectra from a modern perspective~\cite{May:rant}, which has
substantial overlap with some of the review of the classical
foundations in Section~\ref{sec:constr-may-quinn}.

\subsection{Acknowledgments}

We thank Peter May  for his many contributions to this subject and for
useful conversations and correspondence.  We are also very grateful to
Mike Mandell for invaluable help with many parts of this project.
We thank Jacob Lurie for helpful conversations and encouragement.  We
thank John Lind for pointing out an error in a previous draft.  Some
of the results in Section~\ref{sec:ainfty-thom-spectrum} are based on
work in the 2005 University of Chicago Ph.D. thesis of the second
author: he would like to particularly thank his advisors, May and
Mandell, for all of their assistance.

\section{The space of units}

\label{sec:space-units-derived}

In this section, we recall the classical definition of $\GL{R}$ and
explain how to use modern categories of spectra to interpret the units
as a model for the derived space of homotopy automorphisms of the ring 
spectrum $R$.  This preliminary work provides necessary foundations
for our analysis of our new construction of the Thom spectrum functor
in Section~\ref{sec:ainfty-thom-spectrum}.  We do not make any
particular claim to novelty in this section; in particular, May and
Sigurdsson provide an excellent discussion of the situation in
\cite[\S 22.2]{MR2271789} (although note that our use of $\End$ and
$\Aut$ is slightly different than theirs), and the conceptual
description we describe is of course implicit in the original
definition in \cite{MQRT:ersers}.

Given an $\ainfty$ or $\einfty$ ring spectrum $R$ in the classical
sense (e.g., see~\cite{LMS}), the classical construction of the
the group-like $\ainfty$ or $\einfty$ space $\GL{R}$ is as follows:

\begin{Definition}\label{def:gl1}
The space $\GL{R}$ is the pullback in the diagram 
\[
\xymatrix{
\GL{R} \ar[r] \ar[d] & \linf R \ar[d] \\
(\pi_{0}R)^{\times} \ar[r] & \pi_{0}R.
}
\]
(Since the right-hand vertical map is a fibration, the pullback
computes the homotopy pullback.)
\end{Definition}

If $X$ is any space, then
\[
    [X,\GL{R}] = \{f\in R^{0} (\pt{X}) \, |\, \pi_{0}f (X) \subset
(\pi_{0}R)^{\times}\} = R^{0} (\pt{X})^{\times},
\]
which provides a justification for this definition.

We now explain how to interpret $GL{R}$ as the space of homotopy automorphisms of $R$ as an $R$-module.  To
begin, we need to work in a modern
category of spectra, in order to have a sensible category of
$R$-modules.  Assume that $\monspectra$ is a suitable symmetric monoidal
topological model category of spectra, and let $R$ be an $S$-algebra,
i.e., a monoid in $\monspectra$.  The category of $R$-modules inherits
a model structure, and by the space of homotopy automorphisms of $R$,
we mean the subspace of the derived mapping space $\Map_{\Rmod} (R,R)$
consisting of weak equivalences.

In order to make this notion homotopically meaningful, we need to
ensure that the mapping space has the right homotopy type.

\begin{Definition}\label{def-9}
If $R'$ is a cofibrant-fibrant $S$-algebra, and $M$ is a
cofibrant-fibrant $R'$-module, then the \emph{space of endomorphisms}
of $M$ is
\[
   \End (M) \eqdef \Map_{\Mod{R'}}(M,M).
\]
This has a product induced by composition, and by definition the
\emph{space of homotopy automorphisms} of $M$ is the subspace of
group-like components: that is, $\Aut (M) = \GL{\End (M)}$ is the
pullback in the diagram   
\[
\xymatrix{
    \Aut (M) \ar[r] \ar[d] & \ar[d]  \End (M) \\ 
(\pi_{0} (\End (M))^{\times}) \ar[r] & \pi_{0} \End (M).
}
\]
Since $M$ is cofibrant and fibrant, we can equivalently define $\Aut
(M)$ to be the subspace of $\End (M)$ consisting of the homotopy
equivalences.
\end{Definition}
 
If $R$ is an arbitrary algebra, then the \emph{derived space of
endomorphisms} of $R$ is the homotopy type  
\[
\End (R) = \End (\Rcf) \eqdef \Map_{\Mod{R'}} (\Rcf,\Rcf),
\]
where $R'$ is a cofibrant-fibrant replacement of $R$ as an algebra,
and $\Rcf$ is a cofibrant-fibrant replacement of $R'$ as a module over
itself.  The \emph{derived space of homotopy automorphisms} of $R$ is the
homotopy type of the subspace 
\[
\Aut (R) = \Aut (\Rcf) \subset \End (\Rcf)
\]
of homotopy equivalences of $\Rcf$.

In analogy with the notation $\GL{R}$, we have elected to use the
notation $\Aut (R)$ for the space of homotopy automorphisms of $\Rcf$,
even though it is not a strict group.  As defined, we have presented
$\Aut (R)$ as a group-like topological or simplicial monoid.  In
practice, it is easier to access this homotopy type if we let $R^{c}$
be a cofibrant replacement of $R'$, and $R^{f}$ a fibrant replacement.
Then we have a weak homotopy equivalence of spaces 
\[
      \End (R) \heq \Map_{\Mod{R'}}(R^{c},R^{f}),
\]
with $\Aut (R)$ equivalent to the subspace of weak equivalences.  

We now compare $\Aut (R)$ to $\GL{R}$, in the
setting of the $S$-modules of \cite{EKMM}.  Let $\spectra$ be the
Lewis-May-Steinberger category of spectra, let $\spectra[\L]$ denote
the category of $\L$-spectra, let $\EKMM_S$ denote the associated
topological model category of $S$-modules, and write $U \colon
\EKMM_S \to \spectra$ for the forgetful functor.

\begin{Proposition} \label{inf-t-pr-Rwe-and-GL}
Let $R$ be a cofibrant $S$-algebra or commutative $S$-algebra in
$\EKMM_{S}$.  Then there are natural zig-zag of equivalences 
\[
\End(R) \htp \Omega^\infty UR \qquad\textrm{and}\qquad \Aut(R) \htp
\GL{R}
\]
and a zig-zag of natural equivalences between the inclusion of derived
mapping spaces  
\[
  \Aut (R) \rightarrow \End (R)
\]
and the inclusion map
\[
  \GL{R} \rightarrow \linf R.
\]
\end{Proposition}

\begin{proof}
In the model structure on $R$-modules, all objects are fibrant.  Thus,
we can use $R$ for $R^{f}$.  In the notation of \cite{EKMM}, $S \Smash_{\Lin} \L
\sinf S$ is a cofibrant replacement for $S$ as an $S$-module, and
$R\Smash_{S} \L\sinf S$ is a cofibrant replacement for $R$ as an
$R$-module.  So the derived mapping space $\Map_{\EKMM_{R}} (R^{c},R^{f})$ is
given by
\begin{align*}
  \Map_{\EKMM_{R}}(R\Smash_{S}\L \sinf S^{0},R) & \iso \Map_{\EKMM_S} (S\Smash_{S}\L
  \sinf S^{0},R) \\ & \iso \Map_{\alg{\L}{\spectra}} (\L\sinf S^{0},F_{\Lin}
  (S,R)) \\ & \iso \Map_{\spectra} (\sinf S^{0},F_{\Lin} (S,R)) \\ & \iso
  \linf F_{\Lin} (S,R),
\end{align*}
where $F_{\Lin}(-,-)$ denotes the mapping space in $\L$-spectra.

By \cite[\S{I}, Cor 8.7]{EKMM}, the natural map of $\L$-spectra
\[
    R \to F_{\Lin} (S,R)
\]
is a weak equivalence of $\L$-spectra, and so of spectra.
The weak equivalence
\[
   \Map_{\EKMM_{R}} (R\Smash_{S}\L \sinf S^{0},R) \heq \linf R
\]
follows since $\linf$ preserves weak equivalences.  By comparing
pullback diagrams, it is then straightforward to see that the subspace
of $R$-module weak equivalences corresponds to $\GL{R}$.
\end{proof}

The proof of the preceding proposition illustrates how useful it is
that in the Lewis-May-Steinberger and Elmendorf-Kriz-Mandell-May
categories of spectra, an algebra or commutative algebra $R$ is
automatically fibrant as a module over itself, so that $\linf R$
is homotopically meaningful.  In particular, since $\GL{R}$ is
identified as a subspace of $\linf R$, it is evident how to identify
the multiplicative structure on $\GL{R}$.  As we shall 
see in Section~\ref{sec:ainfty-thom-spectrum}, this simplifies our
analysis substantially.

\begin{Remark}
In the setting of a category of diagram spectra $\C$ (e.g., orthogonal
spectra), the situation is somewhat more complicated.  For an associative
$S$-algebra $R$, one can carry out a similar analysis after passing to
a cofibrant-fibrant replacement of $R$ as an $S$-algebra, and the
pullback description of $\GL{R}$ in fact yields a genuine
topological monoid \cite[22.2.3]{MR2271789}.  But the situation for
commutative $S$-algebras in the diagrammatic setting is different.
The model structure on commutative $S$-algebras is lifted from the
positive model structure on (orthogonal) spectra, and in this
model structure the underlying $S$-module of a cofibrant-fibrant
commutative $S$-algebra will not be fibrant; indeed its zero space
will be $S^{0},$ and so    
\[
      \Map_{\C} (S,R) = S^{0} \neq \Map_{\C} (S^{0},R^{f})
      \heq h\End (R).
\]
Of course, one can instead replace the given commutative $S$-algebra
by an associative $S$-algebra instead, but in this case it is
impossible to recover the $\einfty$ structure on $\GL{R}$.  To
describe $\GL{R}$ in this setting requires a different construction;
see \cite{MR2057776} or \cite{Lind} for a description.

The problem that arises above is a manifestation of Lewis's theorem
\cite{MR1124786} about the nature of symmetric monoidal categories of
spectra.  If $S = \sinf S^{0}$ is cofibrant (as it is in diagram
categories of spectra), then the zero space of a cofibrant-fibrant
commutative $S$-algebra must not be homotopically meaningful, as
otherwise we could make a cofibrant-fibrant replacement $S'$ of $S,$
and
\[
    \Map_{\C} (S,S') \heq QS^{0}
\]
would realize $QS^{0}$ as a commutative topological monoid.  On the
other hand, if the zero space of a cofibrant-fibrant commutative
$S$-algebra is homotopically meaningful, then $S$ cannot be cofibrant,
and the $(\sinf,\linf)$ adjunction must take a modified form (as it
does in the setting of Elmendorf-Kriz-Mandell-May spectra).
\end{Remark}

\section{$\ainfty$ Thom spectra and orientations}

\label{sec:ainfty-thom-spectrum}

In this section, we describe a new model of the Thom spectrum functor
and apply it to the study of orientations of $\ainfty$ ring spectra.
The technical foundation of our model is recent work on ``rigid''
models of infinite loop spaces that constructs symmetric monoidal
categories of ``spaces'' such that monoids and commutative monoids
model $A_\infty$ and $E_\infty$ spaces.  There are now several
well-developed categories of rigid spaces, notably $*$-modules, the
space-level analogue of Elmendorf-Mandell-Kriz-May $S$-modules, and
$\catsymbfont{I}$-spaces, the space-level analogue of symmetric
spectra~\cite{Blumberg-Cohen-Schlichtkrull}.

We work with $*$-modules, because the version of the $(\Sigma^\infty,
\Omega^\infty)$ adjunction in this setting is technically felicitous
for dealing with units, as explained in
Section~\ref{sec:space-units-derived}.  The essential strategy is to
adapt the operadic smash product of \cite{MR1361938,EKMM} to the
category of spaces. Specifically, we produce a symmetric monoidal
product on a model of the category $\spaces$ of topological
spaces such that monoids for this product are precisely
$\ainfty$-spaces; in particular, this allows us to work with models of
$\GL{R}$ which are strict monoids for the new product.  The
observation that one could carry out the program of \cite{EKMM} in the
setting of spaces is due to Mike Mandell, and was worked out in the
thesis of the second author \cite{Blumberg-thesis}.  A detailed
presentation of the theory (along with complete proofs) has appeared
in~\cite{Blumberg-Cohen-Schlichtkrull} (and see also~\cite{Lind}).

In order to alleviate the burden on the reader, below we give a very 
streamlined exposition focused on the precise properties we need, with
careful citations.  The results we need that are not in the literature
are proved below.

\subsection{The categories of $\L$-spaces and $*$-modules}

We begin by reviewing the linear isometries operad~\cite[\S
  I.3]{EKMM}.  Fix a countably infinite-dimensional real vector space
$\universe{U}$ topologized as the colimit of its finite-dimensional
subspaces, and let $\Lin(k)$ denote the space of linear isometries
$\universe{U}^{\oplus k} \to \universe{U}$, given the usual function
space topology.  Observe that $\Lin(0)$ is a point and $\Lin(1)$ is a
monoid with unit given by the identity map $\universe{U} \to
\universe{U}$.  Each space $\Lin(k)$ has a free (right) action of
$\Sigma_k$ by permutations and is contractible, and the structure maps
induced from the direct sum of linear isometries make the collection
$\{\Lin(k)\}$ into an $\einfty$ operad.  If we ignore the
permutations, the linear isometries operad is an $A_\infty$ operad.

Let $\spaces$ denote the category of compactly generated weak
Hausdorff spaces.  We define an $\L$-space to be a space with an
action of $\Lin(1)$, and write $\spaces[\L]$ for the category of
$\L$-spaces.  Mimicking the definition of the smash product of
$\bL$-spectra (in the development of Elmendorf-Kriz-Mandell-May), we
have an associative and 
commutative product $X \boxtimes Y$ on the category $\spaces[\L]$~\cite[4.1,4.2]{Blumberg-Cohen-Schlichtkrull} given by the
coequalizer in the diagram  
\[
\xymatrix{ {\Lin (2)\times (\Lin (1)\times \Lin (1)) \times (X \times
Y)} \ar@<.3ex>[rr]^-{\gamma\times 1} \ar@<-.3ex>[rr]_-{1\times \xi} &
& {\Lin (2)\times X \times Y} \ar[r] & {X\timesL Y}  }.
\]
Here $\xi$ denotes the map using the $\L$-algebra structure of $X$ and
$Y,$ and $\gamma$ denotes the operad structure map $\Lin (2)\times
\Lin (1)\times \Lin (1)\to \Lin (2)$.  The left action of $\Lin (1)$
on $\Lin (2)$ induces an action of $\Lin (1)$ on $X\boxtimes Y.$

There is a corresponding internal mapping object 
$F_{\sL}(-,-)$~\cite[4.3,4.4]{Blumberg-Cohen-Schlichtkrull}.
The product is weakly unital, in the sense that for any $\L$-space $X$
there is a natural map $* \boxtimes X \to X$ which is a weak
equivalence~\cite[4.5.,4.6]{Blumberg-Cohen-Schlichtkrull}.

\begin{Theorem}{\cite[4.16]{Blumberg-Cohen-Schlichtkrull}}
The category $\spaces[\L]$ of $\L$-spaces is a topological model category
with weak equivalences the underlying equivalences of spaces.  The
forgetful functor $\spaces[\L] \to \spaces$ is the right adjoint of a Quillen
equivalence.
\end{Theorem}

The product $\boxtimes$ is a version of the cartesian product of
spaces; specifically, for $X$ and $Y$ cofibrant objects of $\T[\L]$
there is a canonical map induced by the universal property of the
product
\[
X \boxtimes Y \to X \times Y
\]
that is a weak equivalence~\cite[4.24]{Blumberg-Cohen-Schlichtkrull}.
This suggests that to study $A_\infty$-spaces one might consider the
category of monoids in $\spaces[\L]$, i.e., the category
$(\spaces[\L])[\bT]$ of  algebras in $\spaces[\L]$ over the associated monad  
\[
\bT X = \bigvee_n \underbrace{X \boxtimes \ldots \boxtimes X}_{n}.
\]
These monoids model $A_\infty$-spaces structured by the linear
isometries operad: 

\begin{Proposition}{\cite[4.8]{Blumberg-Cohen-Schlichtkrull}}
Let $\A$ denote the monad on the category $\spaces$ associated to the
non-symmetric linear isometries operad.  Then the category
$\spaces[\A]$ of $A_\infty$ algebras is equivalent to the category
$(\spaces[\L])[\bT]$.
\end{Proposition}

\begin{Remark}
As one would hope, commutative monoids in $\spaces[\L]$ are $\einfty$
spaces.  However, since we do not need the commutative theory herein,
we have chosen to omit discussion of it.
\end{Remark}

In order to have a symmetric monoidal category, we restrict to the
unital objects.  We define the category $\aM_*$ of $*$-modules to be
the full subcategory of $\L$-spaces such that the unit map $*
\boxtimes X \to X$ is a
homeomorphism~\cite[4.9]{Blumberg-Cohen-Schlichtkrull}.  When
restricted to $\aM_*$, we will continue to write $\boxtimes$ for the
product and $F_{\boxtimes}(-,-)$ for the internal mapping
object $* \boxtimes F_{\sL}(-,-)$.  We then have the following result:

\begin{Theorem}{\cite[4.17]{Blumberg-Cohen-Schlichtkrull}}
The category $\aM_*$ of $*$-modules is a closed symmetric monoidal
topological model category, with product $\boxtimes$, unit $*$,
and internal hom $F_{\boxtimes}(-,-)$.  The weak equivalences are the maps
which are underlying weak equivalences of spaces.  The forgetful
functor $\aM_* \to \spaces$ is the right adjoint of a Quillen equivalence.
\end{Theorem}

All objects in the model structure on $\aM_*$ are
fibrant~\cite[4.18]{Blumberg-Cohen-Schlichtkrull}.  The inclusion
$\aM_* \to \spaces[\L]$ has a right adjoint given by the functor $*
\boxtimes X$.  It is formal that right adjoints on $\spaces[\L]$ can
therefore be lifted to $\aM_*$ by applying this functor.

The monad $\bT$ restricts to $\aM_*$, and the model structure on
$\aM_*$ lifts to a topological model structure on $\aM_*[\bT]$ in
which the weak equivalences and fibrations are determined by the
forgetful functor $\aM_*[\bT] \to
\aM_*$~\cite[4.19]{Blumberg-Cohen-Schlichtkrull}.  

\begin{Lemma}{\cite[4.12]{Blumberg-Cohen-Schlichtkrull}}
Let $M$ be an $A_\infty$ algebra in $\spaces$ over the linear isometries
operad.  Then $* \boxtimes_{\sL} M$ is a monoid in $\aM_*$.
\end{Lemma}

Associated to a monoid $M$ in $\aM_*$ we can consider the category of
modules.  If $G$ is a monoid in $\aM_*$, then a $G$-module is an object
of $\aM_*$ equipped with a map 
\[
    G \boxtimes P \to P
\]
satisfying the usual associativity and unit conditions.  We write
$\aM_G$ for the category of $G$-modules.  

\begin{Theorem}
The category $\aM_G$ is a topological model category in which the
fibrations and weak equivalences are determined by the forgetful
functor $\aM_G \to \aM_*$. 
\end{Theorem}

\begin{proof}
The proof of this result is analogous to the proof
of~\cite[4.16]{Blumberg-Cohen-Schlichtkrull}.  Specifically, we regard
the category $\aM_G$ as the category of algebras in $\aM_*$ over the
monad $G \boxtimes (-)$.  Since $\aM_*$ is a cofibrantly generated
topological model category with all objects fibrant (and satisfying
suitable smallness hypotheses on the generating cofibrations), we can
apply the standard techniques for lifting model structures to monadic
algebras (e.g., see~\cite[4.15]{Blumberg-Cohen-Schlichtkrull}).
\end{proof}

Let $\aA$ denote either the category $\spaces[\L]$ or $\aM_{*}$.  As
usual, we say that a monoid $M$ in $\aA$ is group-like if 
$\pi_0(M)$ is a group.  Let $(\aA[\bT])^{\times}$ denote the full
subcategory of $\aA[\bT]$ consisting of group-like objects.  Because
an $A_\infty$ space is precisely a monoid in $\spaces[\L]$,
Definition~\ref{def:gl1} can be 
interpreted as a functor
\[
\GL \colon \aM_{*}[\bT] \to (\spaces[\L])[\bT]^{\times}.
\]
Composing with $* \boxtimes (-)$ produces a functor
\[
\GL \colon \aM_*[\bT] \to (\aM_*[\bT])^{\times}
\]
which is the right adjoint to the inclusion $(\aM_*[\bT])^{\times} \to
\aM_*[\bT]$.

Given a monoid $M$, we will be interested in the bar construction.
Thus, we will need to employ geometric realization in $\aM_*$.  Given
a simplicial object $X_\bullet$ in $\aM_*$, there is a natural
homeomorphism $U|X_\bullet| \cong
|UX_\bullet|$~\cite[4.26]{Blumberg-Cohen-Schlichtkrull}, where here
$U$ denotes the forgetful functor to spaces.  As in~\cite[\S
  3.1]{Blumberg-Cohen-Schlichtkrull}, we say that a
simplicial object $X_\bullet$ in $\aM_*$ is good if the degeneracies
are $h$-cofibrations in the following sense: a morphism $X \to Y$ in
$\aA$ is an $h$-cofibration if the map 
\[
(X \boxtimes I) \coprod_X Y \to Y \boxtimes I
\]
has a retract.  In this case, the underlying simplicial space
$UX_\bullet$ is good in the classical sense~\cite[\S A]{Segal}.

Given a monoid $M$ in $\aM_*[\bT]$ and right and left $M$-modules $X$ 
and $Y$, we can define the bar construction as the geometric
realization in $\aM_*$ of the simplicial object with $k$-simplices  
\[
B_k(X,M,Y) = X \boxtimes \underbrace{M \boxtimes M \boxtimes \ldots \boxtimes M}_k \boxtimes Y,
\]
with the usual simplicial structure maps induced by the multiplication
on $M$ and the action maps on $X$ and $Y$.  In particular, we define
\[
B_{\boxtimes}G = |B_\bullet(*,G,*)| \qquad\textrm{and}\qquad E_{\boxtimes}G = |B_\bullet(*,G,G)|.
\]
Notice that $E_{\boxtimes}G$ becomes a $G$-module via the action on
the right and the map $\pi \colon E_{\boxtimes}G \to B_{\boxtimes}G$
becomes a map of $G$-modules when we give $B_{\boxtimes}G$ the trivial
action  
\[
B_{\boxtimes} G \boxtimes G \to B_{\boxtimes} G \boxtimes * \to
B_{\boxtimes} G.
\]
By inspection of the definition of
$\boxtimes$~\cite[4.1]{Blumberg-Cohen-Schlichtkrull}, we see
that the fiber at the basepoint of this map is precisely the
realization of the simplicial object with $k$-simplices $G \boxtimes *
\boxtimes \ldots \boxtimes *$, which is homeomorphic to $G$.  Again
let $\aA$ denote one of the categories $\spaces[\L]$ or $\aM_{*}$.  We
say 
that an object of $\aA[\bT]$ is a well-based monoid in $\aA$ if the
unit map $1_{\aA} \to M$ is an $h$-cofibration.  When $M$ is a
well-based monoid, these simplicial objects are
good~\cite[3.2]{Blumberg-Cohen-Schlichtkrull}.

In order to understand the homotopy type of $B_{\boxtimes} G$, we
recall that we have a continuous strong symmetric monoidal functor $Q
\colon \spaces[\L] \to \spaces$ that is the left adjoint to the
functor which gives a space the trivial
$\L$-action~\cite[4.13,4.14]{Blumberg-Cohen-Schlichtkrull}.  The
functor $Q$ comes equipped with a natural transformation $U \to Q$
which is a weak equivalence when applied to cofibrant objects in
$\spaces[\L]$, $\aM_*$, or
$\aM_*[\T]$~\cite[4.27]{Blumberg-Cohen-Schlichtkrull}.  In fact, we
have the following comparison results:

\begin{Theorem}\label{thm:gcomp}
The functor $Q$ induces a Quillen equivalence between $\aM_*$ and
$\spaces$, and a Quillen equivalence between $\aM_G$ and $QG\spaces$
(where the latter is equipped with the standard model structure
determined by the underlying equivalences).
\end{Theorem}

\begin{proof}
The proof of~\cite[4.27]{Blumberg-Cohen-Schlichtkrull} shows that the
left adjoint functor $Q \colon \aM_* \to \spaces$ preserves
cofibrations and weak equivalences between cofibrant objects.
Therefore, $Q$ is a left Quillen adjoint.  Since $Q$ is strong
symmetric monoidal, it lifts to a functor $Q \colon \aM_G \to QG\spaces$.
Since the model structure on $\aM_G$ is lifted from $\aM_*$, an
analogous elaboration of the argument
for~\cite[4.27]{Blumberg-Cohen-Schlichtkrull} shows that $Q$ is a
Quillen left adjoint in this setting as well.  Since the right adjoint
preserves all weak equivalences and $Q$ preserves weak equivalences
between cofibrant objects, in each case $Q$ induces a Quillen
equivalence. 
\end{proof}

Since $Q$ is a continuous left adjoint, it commutes with geometric
realization.  As a particular consequence, we see that $QB_{\boxtimes}
G \cong B QG$, where $B$ denotes the usual bar construction for the
topological monoid $QG$.  Since $U B_{\boxtimes} G \to QB_{\boxtimes}
G$ is a weak equivalence, this identifies the bar construction as the
usual one applied to the rectification $QG$.  This comparison allows
us to show the that the map $E_\boxtimes G \to B_\boxtimes G$ is a
quasifibration in the following sense.

\begin{Theorem}\label{t-EL-BL-quasi-fib}
Let $G$ be a group-like cofibrant object in $\aM[\T]$ which is
well-based.  Then the map $UE_{\boxtimes}G \to UB_{\boxtimes}G$ is a
quasifibration of spaces. 
\end{Theorem}

\begin{proof}
By the remarks above, $QE_{\boxtimes}{G} \cong E(QG)$ and
$QB_{\boxtimes}{G} \cong B(QG)$.  By naturality, there is a
commutative diagram 
\[
\xymatrix{
UE_{\boxtimes}{G} \ar[r] \ar[d]^{U\pi} & E(QG) \ar[d]^{Q\pi} \\
UB_{\boxtimes}{G} \ar[r]^f & B(QG).\\
}
\]
For any $p \in UB_{\boxtimes}{G}$, $(U\pi)^{-1}(p) = UG$; $\pi^{-1}(p)
= G$ by inspection of the definition of
$\boxtimes$~\cite[4.1]{Blumberg-Cohen-Schlichtkrull}.
Furthermore, $(Q\pi)^{-1}(fp) =
QG$, and the map between them is induced from the natural
transformation $U \to Q$.  Writing $F(U\pi)_p$ for the homotopy fiber
of $U\pi$ at $p$ and $F(Q\pi)_{fp}$ for the homotopy fiber of
$Q\pi$ at $fp$, we have a commutative diagram
\[
\xymatrix{
UG \cong (U\pi)^{-1}(p) \ar[r] \ar[d] & F(U\pi)_p \ar[d] \\
QG \cong (Q\pi)^{-1}(fp) \ar[r] & F(Q\pi)_{fp},\\
}
\]
where the horizontal maps are the natural inclusions of the actual
fiber in the homotopy fiber.  The hypotheses on $G$ ensure that
the vertical maps are weak equivalences: on the left, this follows
directly because $G$ is cofibrant, and on the right, we use the
fact that $UE_{\boxtimes}{G} \to QE_{\boxtimes}{G}$ and
$UB_{\boxtimes}{G} \to QB_{\boxtimes}{G}$ are weak 
equivalences since $U$ and $Q$ commute with geometric realization and
all the simplicial spaces involved are proper.  Furthermore, since
$QG$ is a group-like topological monoid with a nondegenerate
basepoint, $Q\pi$ is a quasifibration \cite[7.6]{May:class}, and so
the inclusion of the actual fiber of $U\pi$ in the homotopy fiber of
$U\pi$ is an equivalence.  That is, the bottom horizontal map is an
equivalence.  Thus, we deduce that the top horizontal map is an
equivalence and so that $U\pi$ is a quasifibration.
\end{proof}

As one would expect from the definition of the category $\aM_*$, the
category of Elmendorf-Kriz-Mandell-May $S$-modules is the natural model for the
stabilization.  Specifically, the $(\Sigma^{\infty}_+, \Omega^\infty)$
adjunction on the category $\spectra$ of Lewis-May-Steinberger spectra
and the natural 
equivalence $\sL(1) \ltimes \Sigma^{\infty}_+ X \cong \Sigma^{\infty}_+
(\sL(1) \times X)$ gives rise to an
adjunction $(\Sigma^{\infty}_{\L+}, \Omega^\infty_{\L})$ connecting
$\spaces[\bL]$ and the Elmendorf-Kriz-Mandell-May category of
$\L$-spectra~\cite[7.2]{Lind}. 
To model this in the setting of $*$-modules, for an $S$-modules $M$ we
define  
\[
\Omega^\infty_{S} M = * \boxtimes_{\sL} \Omega^{\infty}_{\L} F_{\sL}(S,M),
\]
where $F_{\sL}(S,M)$ is the mapping $\L$-spectrum~\cite[7.4]{Lind}.
(See also \cite[\S 6]{Basterra-Mandell} for discussion of this
adjunction.)  We then obtain the following result:

\begin{Theorem}\label{thm:*-mod-stable}{\cite[7.5]{Lind}}
There is a strong symmetric monoidal left Quillen functor 
\[
\Sigma^{\infty}_{\L+} \colon \aM_* \to \aM_S.
\]
The corresponding lax symmetric monoidal right Quillen adjoint is 
\[
\Omega^\infty_{S} \colon \aM_S \to \aM_*.
\] 
\end{Theorem}

A consequence of Theorem~\ref{thm:*-mod-stable} is the following
generalization:

\begin{Corollary}\label{cor:G-mod-stable}
The adjunction $(\Sigma^\infty_{\L+}, \Omega^\infty_{S})$ specializes
to a Quillen adjunction 
\[
\Sigma^{\infty}_{\L+} \colon \aM_G \leftrightarrows \aM_{\splus G} \colon
\Omega^\infty_{S}.
\]
\end{Corollary}

\subsection{Thom spectra}
\label{sec:thom-spectra-2}

Assembling adjunctions from the previous section, we have the
following structured version of the adjunction~\eqref{eq:15}:
\begin{equation} \label{eq:Lainfadj}
\xymatrix{ (\aM_*[\T])^{\times} \ar@<.5ex>[r] & {\aM_*[\T]}
\ar@<.5ex>[r]^-{\Sigma^{\infty}_{\L+}} \ar@<.5ex>[l]^-{\GLsym} & {\aM_S[\T]\colon
\GLsym.}  \ar@<.5ex>[l]^-{\linf_{S}} }
\end{equation}

Taking a cofibrant replacement $(\GL{R})^c$ in the category
$\aM_*[\bT]$, we have a composite map of $S$-modules
\[
\gamma \colon \Sigma^{\infty}_{\L+} ((\GL{R})^c) \to
\Sigma^{\infty}_{\L+} \GL{R} \to R, 
\]
where the second map is the counit of the adjunction in
equation~\eqref{eq:Lainfadj}.

Using Corollary~\ref{cor:G-mod-stable}, we conclude the following
lemma:

\begin{Lemma}
The map $\gamma$ gives $R$ the structure of a left
$\Sigma^{\infty}_{\L+} ((\GL{R})^{c})$ module.
\end{Lemma} 

We can now give the definition of the Thom spectrum functor.  For
convenience, assume that $R$ is a cofibrant $S$-algebra.  We will
regard the input as a map 
\[
f \colon X \to B_{\boxtimes} ((\GL{R})^{c})
\]
of $*$-modules; this entails no loss of generality, as we now
explain.  Suppose that we are given more classical input data in the
form of a map of spaces $f \colon X \to B_{\boxtimes}\GL{R}$.  By
adjunction, this is equivalent to a map $f' \colon \L X \to
B_{\boxtimes}\GL{R}$ in $\spaces[\L]$.  Applying $* \boxtimes_{\sL}
(-)$ yields a map of $*$-modules
\[
f'' \colon * \boxtimes_{\sL} \L X \to * \boxtimes_{\sL}
B_{\boxtimes}\GL{R} \cong B_{\boxtimes}\GL{R}.
\]
Finally, we take the (homotopy) pullback in the diagram
\[
\xymatrix{
\tilde{X} \ar[r]^-{\tilde{f}} \ar[d]^{\htp} &
B_{\boxtimes} ((\GL{R})^{c}) \ar[d]^{\htp} \\
\ast \boxtimes_{\sL} \L X \ar[r]^{f''} & B_{\boxtimes}(\GL{R}),
}
\]
where the righthand vertical map is induced by the cofibrant
replacement $(\GL{R})^{c} \to \GL{R}$ in $\aM_*[\bT]$.

\begin{Definition} \label{def-thom-ainfty}
Let $f \colon X \to B_{\boxtimes} ((\GL{R})^{c})$ be a map in $\aM_*$.
The \emph{Thom spectrum} of $f$ is the functor  
\[
M \colon \aM_* / B_{\boxtimes} ((\GL{R})^{c}) \to \aM_R 
\]
given by
\[
Mf  \eqdef \Sigma^{\infty}_{\L+} P' \Smash_{\Sigma^{\infty}_{\L+} ((\GL{R})^{c})} R,
\]
where here $P'$ denotes a cofibrant replacement as a (right) 
$(\GL{R})^{c}$-module of the pullback $P$ in the diagram 
\[
\xymatrix{ 
P \ar[r] \ar[d] & E_{\boxtimes} ((\GL{R})^{c}) \ar[d]
\\
X \ar[r]^-{\tilde{f}} & B_{\boxtimes} ((\GL{R})^{c})
}
\]
(here $\tilde{f}$ is a fibrant replacement of $f$).  
\end{Definition}

By construction, $\Sigma^{\infty}_{\L+} P'$ is then a cofibrant
$\Sigma^{\infty}_{\L+} ((\GL{R})^{c})$-module, so we are computing the
derived smash product.  Moroever, since $R$ is cofibrant $S$-algebra,
the resulting Thom spectrum $Mf$ is a cofibrant $R$-module.

\begin{Remark}
Definition~\ref{def-thom-ainfty} constructs the Thom spectrum directly
as a homotopical functor and a homotopical left adjoint.  One might
hope to construct a point-set Thom functor which we then derive in the
usual fashion, but because this definition involves the composite of a
right adjoint equivalence (the pullback functor from $\aM_* /
B_{\boxtimes} ((\GL{R})^{c})$ to $\aM_{(\GL{R})^{c}}$) and a left adjoint
(the functor $\Sigma^{\infty}_{\L+} (-) \smash_{\Sigma^{\infty}_{\L+}
  ((\GL{R})^{c})} R$ from $\aM_{(\GL{R})^{c}}$ to $\aM_R$), it is involved
(although possible) to give a model which can be derived without the
intermediate cofibrant replacement step.
\end{Remark}

We now want to interpret the notion of orientation in this setting.
We first observe that for any right $R$-module $T$ there is a
natural equivalence of mapping spaces
\[
 \Map_{\aM_R} (Mf,T)\heq
 \Map_{\aM_{\Sigma^{\infty}_{\L+} ((\GL{R})^{c})}} (\Sigma^{\infty}_{\L+}
 P',T) \htp \Map_{\aM_{(\GL{R})^{c}}} (P',\linf_{S} T).
\]
Note that here we are computing derived mapping spaces because all
objects are fibrant in all of the model categories involved.  In
particular, taking $T=R$ we have 
\begin{equation} \label{eq:19}
\Map_{\aM_{R}} (Mf,R) \heq \Map_{\aM_{(\GL{R})^{c}}} (P',\linf_{S} R).
\end{equation}

This gives rise to the following definition of the space of
orientations of a Thom spectrum.

\begin{Definition} \label{def-orientation}
The space of \emph{orientations} of $Mf$ is the subspace of
components of the (derived) mapping space $\Map_{\aM_R} (Mf,R)$ which
correspond to  
\[
\Map_{\aM_{(\GL{R})^{c}}} (P',\GL{R}) \subseteq \Map_{\aM_{(\GL{R})^{c}}}(P',\linf_{S} R)
\]
under the adjunction \eqref{eq:19}.  That is, we form the homotopy
pullback diagram
\begin{equation} \label{eq:71}
\xymatrix{ {\CatOf{orientations} (Mf,R)} \ar[r]^-{\heq}
\ar@{>->}[d] & {\Map_{\aM_{(\GL{R})^{c}}} (P',\GL{R})} \ar@{>->}[d]
\\
{\Map_{\aM_{R}} (Mf,R)} \ar[r]^-{\heq} & \Map_{\aM_{(\GL{R})^{c}}}
(P', \linf_{S}{R}).}
\end{equation}
\end{Definition}

We can provide an obstruction theoretic description of the space of
orientations in terms of lifts in the diagram
\begin{equation} \label{eq:46}
\xymatrix{ {P} \ar[r] \ar[d] & {E_{\boxtimes}{G}} \ar[d]^{\pi}
\\
{X} \ar[r]^{f} \ar@{-->}[ur] & {B_{\boxtimes}{G}.}  }
\end{equation}

\begin{Theorem} \label{t-pr-princ-bundle-obstr-thy}
Suppose that $G$ is a cofibrant group-like monoid in $\aM_*$, and
$f$ is a fibration.  Then there is a natural zigzag of weak
equivalences between the derived mapping space
$\Map_{\aM_*/B_{\boxtimes}G}(f,\pi)$ of lifts in the diagram
\eqref{eq:46} and the derived mapping space
$\Map_{\aM_G}(P,G)$.
\end{Theorem}

\begin{proof}
We will deduce this result from the corresponding result for
group-like monoids (e.g., see~\cite[8.5]{MR2380927}) using the
functorial rectification process provided by the functor $Q$.  

If $G$ is group-like, then $QG$ is a group-like topological monoid
which has the homotopy type of a $CW$-complex and a nondegenerate
basepoint.  Therefore, applying $Q$ and taking the homotopy pullback,
we obtain a square of $QG$-spaces in $\spaces$
\[
\xymatrix{
\hat{P} \ar[r] \ar[d] & B(QG) \ar[d]\\
QX \ar[r] & E(QG).
}
\]
such that there is a weak equivalence of derived mapping spaces
\[
\Map_{\spaces / B(QG)}(QX, E(QG))\heq\Map_{(QG) \spaces}(\hat{P}, QG).
\]
We now use Theorem~\ref{thm:gcomp}.  On the one hand, a
straightforward extension of Theorem~\ref{thm:gcomp} implies that $Q$
induces a Quillen equivalence between $\aM_*/B_{\boxtimes}G$ and
$\spaces/B (QG)$, and so there is an equivalence of derived mapping spaces 
\[
\Map_{\aM_* / B_{\boxtimes}{G}}(X, E_{\boxtimes}{G}) \heq \Map_{\spaces /
B(QG)}(QX, E(QG)).
\]
On the other hand, since $Q$ also induces a Quillen equivalence
between $\aM_G$ and $QG\spaces$, there is an equivalence of derived
mapping spaces 
\[
  \Map_{\aM_G}(P, G) \to \Map_{QG \spaces}(QP, QG).
\]
The proof of the theorem will be complete once we have shown that a
cofibrant replacement of $QP$ is naturally weakly equivalent to a
cofibrant replacement of $\hat{P}$ as $QG$-spaces.  Finally, this
follows because either derived functor associated to a Quillen
equivalence preserves homotopy limits up to a zigzag of natural weak
equivalences.  Although this result is standard, the authors are not
aware of a convenient reference and so we briefly remind the reader of
the proof.  The homotopy limit of shape $D$ in the homotopical
category $\aC$ is the right derived functor $\Ho(\aC^{D}) \to
\Ho(\aC)$ of the right adjoint (which exists on the level of
homotopical categories) of the constant diagram functor.
Since equivalences of homotopical categories (or Quillen equivalences of cofibrantly generated model categories) induce equivalences on diagram categories (or Quillen equivalences of the projective model structure on the diagram categories),
the result follows by lifting the isomorphism in the homotopy
category to a weak equivalence between cofibrant-fibrant objects. 
\end{proof}

Theorem \ref{t-pr-princ-bundle-obstr-thy} now has the following
immediate corollary.

\begin{Theorem} \label{t-th-orientation-ainfty}
The space of orientations of $Mf$ is weakly equivalent to the space
of lifts in the diagram \eqref{eq:46}.  In particular, the spectrum
$Mf$ is orientable if and only if $f \colon X\to
B_{\boxtimes} ((\GL{R})^{c})$ is null homotopic.
\end{Theorem}

Since our construction of the Thom spectrum takes homotopic
classifying maps to weakly equivalent spectra,
Theorem~\ref{t-th-orientation-ainfty} implies that an orientation
gives rise to an equivalence $Mf \htp \Sigma^{\infty}_{\L+} X \smash
R$.  This is a version of the Thom isomorphism theorem, and we will 
give a description of the map inducing this equivalence below.

\subsection{Orientations and the Thom isomorphism}

To make contact with familiar notions of orientation, we'll be more
explicit about the adjunctions in Definition~\ref{def-orientation}.
For this it it helpful to recapitulate some classical computations of
Thom spectra in our setting.
  
\begin{Lemma} \label{t-le-thom-spectrum-point}
The Thom spectrum of the inclusion of a point
\[
    * \to B_{\boxtimes} ((\GL{R})^{c})
\] 
is a cofibrant $R$-module which is weakly equivalent to $R$.
More generally, the Thom spectrum of a trivial map
\[
X \to * \to B_{\boxtimes} ((\GL{R})^{c})
\]
is weakly equivalent to $R \smash \Sigma^{\infty}_{\L+} X$.
\end{Lemma}

\begin{proof}
Let $\ast  \to B_{\boxtimes} ((\GL{R})^{c})$ be the inclusion of a point.
The Thom spectrum is $\splus P' \Smash_{\Sigma^{\infty}_{\L+}
  (\GL{R})^{c}} R,$ where $P'$ is a cofibrant replacement of the
homotopy pullback
\[
\xymatrix{
P \ar[r]\ar[d] & E_{\boxtimes} ((\GL{R})^{c}) \ar[d]\\
 \ast \ar[r] & B_{\boxtimes} ((\GL{R})^{c}),
}
\]
as a $(\GL{R})^{c}$-module in $\aM_{*}$.  Since
Theorem~\ref{t-EL-BL-quasi-fib} implies that
$UE_{\boxtimes} ((\GL{R})^{c})\to UB_{\boxtimes} ((\GL{R})^{c})$ is a
quasifibration (with fiber $(\GL{R})^{c}$), it follows that
$(\GL{R})^{c} \heq P'$ as $(\GL{R})^{c}$-modules.

Consideration of the iterated pullback square
\[
\xymatrix{
\tilde{P} \ar[r] \ar[d] & \ar[d] P \ar[r]\ar[d] & E_{\boxtimes} ((\GL{R})^{c}) \ar[d]\\
X \ar[r] & \ast \ar[r] & B_{\boxtimes} ((\GL{R})^{c}),
}
\]
implies that $\tilde{P}$ is equivalent to $(\GL{R})^{c} \boxtimes X$ as a 
$(\GL{R})^{c}$-module, where $X$ has the trivial action.
\end{proof}

In particular, we have the following corollary.

\begin{Corollary}\label{t-co-Th-ELGL}
Since $E_{\boxtimes} ((\GL{R})^{c})\heq \ptspace$, we have
\[
   M (\pi \colon E_{\boxtimes} ((\GL{R})^{c}) \to
   B_{\boxtimes} ((\GL{R})^{c}) \heq R.
\]
as $R$-modules.
\end{Corollary}

Now suppose that $f\colon X\to B_{\boxtimes} ((\GL{R})^{c})$ is a fibration of $*$-modules,
and let $P$ be the pullback in the diagram
\begin{equation} \label{eq:72}
\xymatrix{ {P} \ar[r] \ar[d] & {E_{\boxtimes} ((\GL{R})^{c})} \ar[d]^{\pi}
\\
{X} \ar[r]^-{f} \ar@{-->}[ur]^{\tilde{a}} & {B_{\boxtimes} ((\GL{R})^{c}),}  }
\end{equation}
and let $M = Mf$.  If $\tilde{a}$ is a lift as indicated,
then by functoriality passing to Thom spectra along $\tilde{a}$
induces a map of $R$-modules 
\[
      a\colon Mf \to R.
\]
This is the orientation associated to the lift $\tilde{a}$.

Conversely, suppose that $a \colon Mf \to R$ is a map of $R$-modules.
Each point $p\in P'$ (the cofibrant replacement of $P$ as a
$(\GL{R})^{c}$-module) determines a $(\GL{R})^{c}$-map
\[
 (\GL{R})^{c} \to P'
\]
and therefore a map of $\Sigma^{\infty}_{\L+} ((\GL{R})^{c})$-modules 
\[
\Sigma^{\infty}_{\L+} ((\GL{R})^{c}) \to \Sigma^{\infty}_{\L+} P'.
\] 
Passing to Thom spectra, this in turn yields a map of $R$-modules
\[
j_{p}\colon \splus ((\GL{R})^{c}) \Smash_{\Sigma^{\infty}_{\L+}
  ((\GL{R})^{c})} R \cong R \to Mf \to R.
\]
As $p$ varies the $j_{p}$ assemble; we take the adjoint of the
composite
\begin{align*}
\Sigma^{\infty}_{\L+} P' & \longrightarrow F_{\Sigma^{\infty}_{\L+} ((\GL{R})^{c})}(\Sigma^{\infty}_{\L+} ((\GL{R})^{c}), \Sigma^{\infty}_{\L+} P')
\\ & \longrightarrow F_R(\Sigma^{\infty}_{\L+} ((\GL{R})^{c}) \smash_{\Sigma^{\infty}_{\L+} ((\GL{R})^{c})} R,
\splus P' \smash_{\Sigma^{\infty}_{\L+} ((\GL{R})^{c})} R) \\
& \cong F_R(R,Mf) \xrightarrow{a} F_R(R,R),  
\end{align*}
where the first map is the adjoint of the action map and the second
map is induced by functoriality.

The argument for Proposition \ref{inf-t-pr-Rwe-and-GL} shows that
\[
\Omega^{\infty}_{S} F_{R} (R,R)\heq \linf_{S} R,
\]
and the resulting map 
\[
   j\colon P' \to \linf_{S} R
\]
corresponds to $a$ under the equivalence of derived mapping spaces
\[
  \Map_{\EKMM_R} (Mf,R) \heq \Map_{\EKMM_{\GL{R}}} (P',\linf_{S} R).
\]

Put another way, for each $q\in X$,
Lemma~\ref{t-le-thom-spectrum-point} implies that 
the Thom spectrum $M_{q}$ of $ q \to X
\to B_{\boxtimes}\GL{R}$ is non-canonically weakly equivalent to $R$. 
Passing to Thom spectra gives a map
\[
   i_{q}\colon M_{q} \to Mf \xra{a} R.
\]
A choice of point $p\in P$ lying over $q$ fixes an equivalence $R
\heq M_{q}$ making the diagram 
\[
\xymatrix{
R
\ar[dr]_{j_{p}}
\ar[rr]^{\heq}
&
&
{M_{q}}
 \ar[dl]^{i_{q}} \\
&
{R}
}
\]
commute.
Thus we have the following analogue of the standard description of
Thom classes as in for example~\cite[Definition 14.5]{MR0385836}

\begin{Proposition} \label{t-pr-characterize-orientations}
Suppose that $a\colon Mf \to R$ is a map of $R$-modules.  Then the
following are equivalent.
\begin{enumerate}
\item $a$ is an orientation.
\item For each $q\in X$, the map of
$R$-modules
\[
 i_{q}\colon  M_{q}\to Mf \xra{a} R
\]
is a weak equivalence.
\item For each $p\in P$, the map of $R$-modules
\[
j_{p}\colon     R \to Mf \xra{a} R
\]
is a weak equivalence.
\end{enumerate}
\end{Proposition}

We conclude by discussing the Thom isomorphism in this setting.
Let $f \colon X \to B_{\boxtimes} (\GL{R})^{c}$ be a fibration of
$*$-modules and suppose that $X$ is cofibrant in $\aM_*$.  Now suppose
we are given an orientation in the form of a $(\GL{R})^{c}$-map 
\[
s \colon P'\to (\GL{R})^{c},
\]
corresponding to an $R$-module map 
\[
     a\colon Mf \to R.
\]

Consider the map
\[
\xymatrix{
X \boxtimes X \ar[r]^-{f\pi_2} & B_{\boxtimes} (({\GL{R}})^{c}),
}
\]
where here $\pi_2$ is the projection onto the second factor (induced
from the composite $X \boxtimes X \to * \boxtimes X$.  Passing to
pullbacks, we obtain the commutative diagram
\[
\xymatrix{
\tilde{P} \ar[r] \ar[d] & P \ar[r] \ar[d] &
E_{\boxtimes} ((\GL{R})^{c}) \ar[d] \\
X \boxtimes X \ar[r] & X \ar[r] & B_{\boxtimes} (({\GL{R}})^{c}).
}
\]
Since the map $P \boxtimes X \to X \boxtimes X$ induced from the map
$P \to X$ and the projection map $P \boxtimes X \to P$ are compatible
with the maps to $X$, the universal property of the pullback induces a
map $P \boxtimes X \to \tilde{P}$.  Passing to cofibrant replacements
as $(\GL{R})^{c}$-modules gives us a map between cofibrant-fibrant
$(\GL{R})^{c}$-modules; using $Q$ and the argument for
Theorem~\ref{t-pr-princ-bundle-obstr-thy}, we see that this map
represents the identity map on $QX \times QP$ in the homotopy category
of $Q ((\GL{R})^{c})$-spaces, and hence is a weak equivalence.

Let $P'$ denote a cofibrant replacement of $P$ as a
$(\GL{R})^{c}$-module.  Since $P'$ and $X \boxtimes P'$ are
cofibrant-fibrant objects, we can choose a map $P' \to X \boxtimes P'$
which lifts the homotopy class of the diagonal map $QP' \to QX \times
QP'$.  Passing to Thom spectra, we obtain the $R$-module Thom diagonal
map  
\[
M\xra{\Delta} (\Sigma^{\infty}_{\L+} X) \Smash M.
\]

Next, we form the composite 
\begin{equation}\label{eq:thomdiageq}
    M \xra{\Delta} (\Sigma^{\infty}_{\bL+} X) \Smash M \xra{1 \Smash a} (\Sigma^{\infty}_{\L+} X)
    \Smash R
\end{equation}
as in~\cite{MR609673}.  

To analyze this, we compose the orientation $s$ with the map $P' \to P' \boxtimes
X$ to obtain the composite map of $(\GL{R})^{c}$-modules 
\[
P' \to X \boxtimes P' \to X \boxtimes (\GL{R})^{c}.
\]
Now, applying the functor $(-) \smash_{\Sigma^{\infty}_{\L+}
  ((\GL{R})^{c})} R$ produces the Thom diagonal
equation~\eqref{eq:thomdiageq}.  On the other hand, since $s$
corresponds to a section of the map $P \to X$ induced by the universal
property of the pullback, this composite is a weak equivalence of
$(\GL{R})^{c}$-modules.  Since $(-) \smash_{\Sigma^{\infty}_{\L+}
  ((\GL{R})^{c})} R$ preserves weak equivalences of cofibrant
$(\GL{R})^{c}$-modules, we obtain the following proposition:

\begin{Proposition} \label{t-pr-thom-iso}
If $a\colon Mf\to R$ is an orientation, then the map of right $R$-modules
\[
     Mf \xra{\Delta} \Sigma^{\infty}_{\L+} X \Smash Mf \xra{1 \Smash
       a} \Sigma^{\infty}_{\L+} X \Smash R
\]
is a weak equivalence.
\end{Proposition}

\section{$\einfty$ Thom spectra and orientations}
\label{sec:thom-spectra}

In this section, we describe the construction and orientation of
$\einfty$ Thom spectra, generalizing the perspective of
Section~\ref{sec:ainfty-thom-spectrum}.  For an $\einfty$ ring spectrum
$R$, the space of units $\GL{R}$ can be delooped to form a spectrum of
units $\gl{R}$.  This is encoded in the basic adjunction
\begin{equation} \label{eq:5}
\xymatrix{ {\splus \linf\colon \Ho \CatOf{$(-1)$-connected spectra}}
\ar@<.3ex>[r] & {\Ho \einftyspectra \colon \glsym} \ar@<.3ex>[l] }
\end{equation}
which is proved in \S\ref{sec:constr-may-quinn}; see Theorem
\ref{t-th-units}.  Here 
$\einftyspectra$ denotes the model category of Lewis-May-Steinberger
$E_\infty$ ring spectra. 
In order to support the generalization to $R$-algebras, we model
$\einftyspectra$ via the Quillen equivalent model category $\aM_S[\bP]$ of
Elmendorf-Mandell-Kriz-May commutative $S$-algebras, the connective
spectra as a subcategory of $\aM_S$, and $\splus \linf$ as the composite
$\Sigma^{\infty}_{\bL+} \Omega^\infty_S$ (see Theorem \ref{thm:*-mod-stable}).

We begin by discussing the classical case of stable spherical
fibrations.  The counit of the adjunction above yields a map in $\Ho \aM_S[\bP] \cong \Ho
\einftyspectra$
\[
   \epsilon\colon \Sigma^{\infty}_{\bL+} \linf_S \gl{S} \to S.
\]
Assume we are given a map
\[
    \zeta\colon b\to b\gl{S},
\]
where we write $b\gl{S}$ for $\Sigma \gl{S}$.  
Let $j=\Sigma^{-1}\zeta$ and form the
diagram
\[
\xymatrix{ {g} \ar[r]^-{j} \ar[d] & {\gl{S}} \ar@{=}[r] \ar[d] &
{\gl{S}} \ar[d]
\\
{\ptspace} \ar[r] & {Cj} \ar[r] \ar[d] & {e\gl{S} \heq \ptspace}
\ar[d]
 \\
& {b} \ar[r] & {b\gl{S}} }
\]
by requiring that the upper left and bottom right squares are homotopy
Cartesian.  Note that we may also view $b$ as an infinite loop map
\[
     f\colon B \xra{} B\GL{S}.
\]

\begin{Definition}
The \emph{Thom spectrum} of $f$, or of
$\zeta$, or of $j$, is the homotopy pushout $M\zeta$ of the diagram in $\aM_S[\bP]$ 
\begin{equation} \label{or-eq:8}
\xymatrix{
\Sigma^\infty_{\bL+} \linf_S g \ar[r]^-{\Sigma^{\infty}_{\bL+} \linf_S j} \ar[d]^-{\Sigma^\infty_{\bL+} \linf_S \ptspace}
& \Sigma^\infty_{\bL+} \linf_S \gl{S} \ar[r]^-{\epsilon} \ar[r] \ar[d] & S \ar[d] \\ 
S \heq \Sigma^{\infty}_{\bL+} \linf_S \ptspace \ar[r] & \Sigma^{\infty}_{\bL+} \linf_S Cj \ar[r] & M\zeta,
}
\end{equation}
which is to say that
\[
       M\zeta \cong \Sigma^{\infty}_{\bL+} \linf_S Cj
         \Lsmash_{\Sigma^{\infty}_{\bL+} \linf_S \gl{S}} S \cong S \Lsmash_{\Sigma^\infty_{\bL+} \linf_S g} S.
\]
\end{Definition}

Note that the left-hand square in the diagram \eqref{or-eq:8} is a homotopy pushout by definition of $Cj$ and the fact that $\Sigma^\infty_{\bL+}\linf_S$ preserves homotopy pushouts.
Also note that when writing this homotopy pushout, we are suppressing the
choice of a point-set representative of the homotopy class $\epsilon$.
Since all objects are fibrant in the model structure on $\aM_S[\bP]$,
it suffices to choose a cofibrant model for $\linf_S \gl{S}$ (and
subsequently of $\linf_S g$) in the model structure on
$\aM_*[\bP]$~\cite[4.19]{Blumberg-Cohen-Schlichtkrull}.

Now suppose that $R$ is a commutative $S$-algebra with unit
$\iota\colon S\to R$; let $i=\gl{\iota}$, and let $k = ij\colon g\to
\gl{R},$ so that we have the solid arrows of the diagram
\begin{equation}  \label{eq:10}
\xymatrix{ {g} \ar[r]^-{j} \ar[dr]_{k} & {\gl{S}} \ar[r] \ar[d]_{i} &
{Cj} \ar@{-->}[dl]^{u}
\\
& {\gl{R},} }
\end{equation}
in which the row is a cofiber sequence.   The homotopy pushout diagram
\eqref{or-eq:8} and the adjunction \eqref{eq:5} give
the following. 

\begin{Theorem} \label{t-obstruction}
The derived mapping space $\Map_{\aM_S[\bP]}(M\zeta,R)$ is equivalent
to the fiber of the map of derived mapping spaces 
\[
     \Map_{\M_S}(Cj,\gl{R}) \to \Map_{\aM_S}(\gl{S},\gl{R}) 
\]
over the basepoint associated to the map $i: \gl{S}\to \gl{R}$.  That is, the map $k$ is the obstruction to
the existence of an $\einfty$ map $M\zeta \to R$, and
$\Map_{\aM_S[\bP]} (M\zeta,R)$ is weakly equivalent to the space of
lifts in the diagram \eqref{eq:10}.
\end{Theorem}

We have the following $E_\i$ analogue of the usual Thom isomorphism:

\begin{Theorem}
If $\Map_{\aM_S[\bP]}(M\zeta,R)$ is non-empty (i.e. if $k$ is
homotopic to the trivial map $g\to\gl{R}$) then we have equivalences
of derived mapping spaces 
\[
\Map_{\aM_S[\bP]}(M\zeta,R)\heq\Omega\Map_{\aM_S}(g,\gl{R})\heq
\Map_{\aM_S}(b,\gl{R}) \heq \Map_{\aM_S[\bP]}(\Sigma^\infty_{\bL+}
B,R).
\]
\end{Theorem}

More generally, suppose that $R$ is a commutative $S$-algebra.  There
is a category $\aM_R[\bP]$ of commutative $R$-algebras; the functor $R
\smash_S (-)$ is the left adjoint of a Quillen pair connecting
$\aM_R[\bP]$ and $\aM_S[\bP]$ (with right adjoint the forgetful
functor).  Therefore, we can consider the homotopical adjunction $(R 
\smash \Sigma^\infty_{\bL+} \Omega_S U, \gl U)$ connecting $\aM_R$ and
$\aM_R[\bP]$, where here $U$ denotes both the forgetful functor $\aM_R
\to \aM_S$ and $\aM_R[\bP] \to \aM_S[\bP]$ respectively.  In further
abuse of notation, we will suppress $U$ and write $\gl{R}$ for
$\gl{UR}$ and $\Sigma^\infty_{\bL+} \Omega_S$ for
$\Sigma^{\infty}_{\bL+} \Omega_S U$.

Now given a map
\[
   \zeta\colon b \to b\gl{R},
\]
we obtain a map of cofiber sequences
\[
\xymatrix{
g\ar[r]\ar[d] & \gl{R}\ar[d]\\
{\gl{R}} \ar@{=}[r] \ar[d] & {\gl{R}} \ar[d]
 \\
{p} \ar[r] \ar[d] \ar@{-->}[ur] & {e\gl{R}\heq \ptspace } \ar[d]
 \\
{b} \ar[r]^-{\zeta} \ar@{-->}[ur] & {b\gl{R}.}  }
\]
in which $g=\Sigma^{-1} b$ and $p$ is the fiber of $b\to b\gl{R}$.

\begin{Definition} \label{def-4}
The \emph{$R$-algebra Thom spectrum} of $\zeta$ is the commutative
$R$-algebra $M\zeta$ defined as the homotopy pushout in
$\aM_R[\bP]$ of the diagram
\[
\xymatrix{
R \smash \Sigma^{\infty}_{\bL+} \linf_S g \ar[r] \ar[d] &  R \smash
\Sigma^{\infty}_{\bL+} \linf_S \gl{R} \ar[d] \ar[r] & R \ar[d] \\ 
R\smash \Sigma^{\infty}_{\bL+} \linf_S \ast \ar[r] & R\smash
\Sigma^{\infty}_{\bL+} \linf_S p \ar[r] & M\zeta.
}
\]
\end{Definition}
Again, note that the left-hand square is automatically a homotopy pushout in $\aM_R[\bP]$, which means that $M\zeta$ can be taken to be the homotopy pushout of the right-hand square or of the composite square.

The Thom $R$-algebra is a generalization of the Thom $R$-module of Definition~\ref{def-thom-ainfty}:

\begin{Lemma}
The underlying $R$-module of the $R$-algebra Thom spectrum
of $\zeta$ is weakly equivalent to the $A_\infty$ Thom spectrum of
$\Omega^\infty \zeta$.
\end{Lemma}

\begin{proof}
This follows from a check of the definitions given the fact that the
homotopy pushout
\[
\xymatrix
{
B & A \ar[r] \ar[l] & C \\
}
\]
in the category $\aM_R[\bP]$ is naturally weakly equivalent to the
derived smash product $B \Lsmash_{A} C$~\cite[\S VII.1.6]{EKMM}.
\end{proof}

\begin{Theorem} \label{t-obstruction-R}
Let $A$ be a commutative $R$-algebra, and write
$$
i\colon \gl{R}\to\gl{A}
$$
for the induced map on unit spectra.  The derived mapping space
$\Map_{\aM_R[\bP]}(M\zeta,A)$ is weakly equivalent to the fiber in the
map of derived mapping spaces
\begin{equation}
\Map_{\aM_S} (p,\gl{A}) \to \Map_{\aM_S} (\gl{R},\gl{A})
\end{equation}
at the basepoint associated to the map $i$. 
\end{Theorem}

Taking $A=R$, we see that the space of $R$-algebra orientations of
$M\zeta$ is the space of lifts
\[
\xymatrix{ & {e\gl{R}} \ar[d]
\\
{b} \ar[r]_-{\zeta} \ar@{-->}[ur] & {b\gl{R}.}  }
\]
In this form the obstruction theory is a generalization of the
obstruction theory for orientations of $\ainfty$ ring spectra in
Theorem \ref{t-th-orientation-ainfty}.

To make contact with the classical situation, let $S$ be the sphere 
spectrum, and suppose we are given a map
\[
  g\colon b\to b\gl{S},
\]
so that $\linf g$ classifies a stable spherical fibration.  

Now suppose that $R$ is a commutative $S$-algebra with unit
$\iota\colon S\to R$, and let
\[
  f = b\gl{\iota}\circ  g\colon b \to b\gl{S} \to b\gl{R}.
\]
Then
\[
    Mf \heq Mg \Lsmash R,
\]
and so extension of scalars induces an equivalence of derived mapping
spaces
\[
    \Map_{\aM_S[\bP]} (Mg,R) \heq \Map_{\aM_R[\bP]} (Mf,R).
\]
If we let $b (S,R)$ be the homotopy pullback in the solid diagram
\begin{equation} \label{eq:34}
\xymatrix{ {p} \ar[r] \ar[d] & {b (S,R)} \ar[r] \ar[d] & {\ptspace}
\ar[d]
\\
{b} \ar[r] \ar@{-->}[ur] & {b\gl{S}} \ar[r] & {b\gl{R},} }
\end{equation}
then Theorem~\ref{t-obstruction-R} specializes to a result of May,
Quinn, Ray, and Tornehave \cite{MQRT:ersers}.

\begin{Corollary}\label{t-thom-iso-mqr}
The derived space of $\einfty$ maps $Mg\to R$ is weakly equivalent to
the derived space of lifts in the diagram \eqref{eq:34}.
\end{Corollary}

\section{Units after May-Quinn-Ray}

\label{sec:constr-may-quinn}

Our construction of the Thom spectrum in
\S\ref{sec:ainfty-thom-spectrum} uses a model for the
adjunction   
\[
\xymatrix{{\CatOf{group-like $A_{\infty}$ spaces}} \ar@<.5ex>[r] &
{\ainftyspaces} \ar@<.5ex>[r]^-{\splus} \ar@<.5ex>[l]^-{\GLsym} &
{\CatOf{$\ainfty$ ring spectra} \colon \GLsym,} \ar@<.5ex>[l]^-{\linf} }
\]
which is a homotopical refinement of the standard adjunction 
\[
\xymatrix{ {\Z\colon \CatOf{groups}} \ar@<.3ex>[r] & {\CatOf{rings}.}  \ar@<.3ex>[l] }
\]
For the $\einfty$ case we use the $\einfty$ analog, 
\[
\xymatrix{{\CatOf{group-like $E_{\infty}$ spaces}} \ar@<.5ex>[r] &
{\CatOf{$\einfty$ spaces}} \ar@<.5ex>[r]^-{\splus} \ar@<.5ex>[l]^-{\GLsym} &
{\CatOf{$\einfty$ ring spectra}\colon \GLsym,} \ar@<.5ex>[l]^-{\linf} }
\]
which is modeled on the analogous adjunction 
\[
\xymatrix{ {\Z\colon \CatOf{abelian groups}} \ar@<.3ex>[r] &
{\CatOf{commutative rings}.}  \ar@<.3ex>[l] } 
\]
When $A$ is an $\einfty$ ring spectrum, $\GL{A}$ is a group-like
$\einfty$ space.  Since group-like $\einfty$ spaces model connective
spectra, it follows that there is a spectrum $\gl{A}$ such that 
\begin{equation} \label{eq:16}
     \linf \gl{A}\heq \GL{A}.
\end{equation}

In this section, we give a precise model of the adjunction and combine 
with a modernized version of the delooping result to prove the
following theorem:

\begin{Theorem} \label{t-th-units}
The functors $\splus \linf$ and $\glsym$ induce adjunctions
\begin{equation}
\xymatrix{ {\splus \linf\colon \Ho \CatOf{$(-1)$-connected spectra}}
\ar@<.3ex>[r] & {\Ho \einftyspectra\colon \glsym} \ar@<.3ex>[l] }
\end{equation}
of categories enriched over the homotopy category of spaces.
\end{Theorem}

Note that the construction of this adjunction realizes the left
adjoint as a composite of left Quillen adjoints and Quillen
equivalences and the right adjoint as a composite of right Quillen
adjoints and Quillen equivalences.  As a consequence, the left adjoint
preserves homotopy colimits and the right adjoint preserves homotopy
limits.

\begin{Remark}
In fact, Theorem~\ref{t-th-units} can be formulated as an adjunction
of $\i$-categories 
\[
\xymatrix{ {\splus \linf\colon \CatOf{$(-1)$-connected spectra}}
\ar@<.3ex>[r] & {\einftyspectra\colon \glsym} \ar@<.3ex>[l] }.
\]
See the companion paper~\cite{ABGHR} and the subsequent
paper~\cite{ABG} for a description of such an approach to the Thom
spectrum functor.
\end{Remark}

Throughout this section, we work in the classical categories
$\spectra$ of Lewis-May-Steinberger spectra~\cite{LMS} and
$\einftyspectra$ of $E_\infty$ ring spectra.  As we noted in
Section~\ref{sec:thom-spectra}, it is often useful to restate this
adjunction using modern models for these homotopy categories.
Since composition with an equivalence of categories preserves the
property of being a left or right adjoint, such a shift is harmless.

The reader will notice that a proof of Theorem~\ref{t-th-units} can
mostly be assembled from results scattered in the literature, particularly
\cite{May:gils,MR0339152,MQRT:ersers,LMS,EKMM}.  We wrote this section 
in order to consolidate this material and in order to present
modernized treatments using the language of model categories.  

\begin{Remark}
We note that May has prepared a review of the relevant
multiplicative infinite loop space theory \cite{May:rant} which also
includes the results we need.
\end{Remark}

\subsection{$\einfty$ spectra} \label{sec:einfty-spectra}

In this section we review the notion of a $C$-spectrum, where $C$ is
an operad (in spaces) over the linear isometries operad.  We also
recall the fact that the homotopy category of $\einfty$ spectra is
well defined, in the sense that if $C$ and $D$ are two $\einfty$
operads over the linear isometries operad, then the categories of
$C$-spectra and $D$-spectra are connected by a zig-zag of continuous
Quillen equivalences.

If $C$ is an operad, then for $k\geq 0$ we write $C (k)$ for the
$k^\text{th}$ space of the operad.   We also write $C$ for the
associated monad.  Let $\spectra = \spectra_{\universe{U}}$ denote the
category of spectra based on a universe $\universe{U}$, in the sense
of \cite{LMS}.  Let $\Lin$ denote the linear isometries operad of
$\universe{U}$, and let $C\to \Lin$ be an operad over $\Lin$.  Then
\[
     CV = \bigvee_{k\geq 0} C (k)\ltimes_{\Sigma_{k}} V^{\wedge k}.
\]
is the free $C$-algebra on $V$.  We write $\alg{C}{\spectra}$
for the category of $C$-algebras in $\spectra$, and we call its
objects $C$-spectra.

In general $C (\ptspace) \iso \splus C (0)$ is the initial object of
the category of $C$-spectra.  We shall say that $C$ is \emph{unital}
if $C (0) = \ptspace,$ so that $C (0)\iso S$ is the sphere spectrum.

Lewis-May-Steinberger work with unital operads and the free
$C$-spectrum with prescribed unit.  If $S\to V$ is a spectrum under
the sphere, then we write $\ptC V$ for the free $C$ spectrum
on $V$ with unit $\iota\colon S\to V \to \ptC V.$ This is the pushout in
the category of $C$-spectra in the diagram
\begin{equation} \label{eq:41}
\xymatrix{
CS \ar[r] \ar[d]^{C \iota} & \ar[d] C (\ptspace )  = S \\
CV \ar[r] & \ptC V.
}
\end{equation}
By construction, $\ptC$ participates in a monad on the category
$\spectra_{S/}$ of spectra under the sphere spectrum.

As explained in \cite[II, Remark 4.9]{EKMM},
\[
    \Svee (V) = S \vee V
\]
defines a monad on $\spectra,$ using the fold map $S \vee S \to S$, 
and we have an equivalence of categories
\[
   \spectra_{S/} \iso \alg{\Svee}{\spectra}.
\]
It follows that there is a natural isomorphism
\begin{equation}\label{eq:44}
       C (V) \iso C_{*} \Svee (V)
\end{equation}
and (\cite[II, Lemma 6.1]{EKMM}) an equivalence of categories
\[
   \alg{C}{\spectra} \iso \alg{C_{*}}{\spectra_{S/}}.
\]

We recall the following, which can be proved easily using the argument
of \cite{EKMM,MR1806878}, in particular an adaptation of the
``Cofibration Hypothesis'' of \S{VII} of \cite{EKMM}.

\begin{Proposition} \label{t-pr-C-spectra-level-model}
The category of $C$-spectra has the structure of a cofibrantly
generated topological model category, in which the forgetful
functor to $\spectra$ creates fibrations and weak equivalences.  If
$\{A\to B \}$ is a set of generating (trivial) cofibrations of
$\spectra$, then $\{CA\to CB \}$ is a set of generating (trivial)
cofibrations of $\alg{C}{\spectra}$.
\end{Proposition}

In particular, the category of $C$-spectra is cocomplete (this is
explained on pp.  46---49 of \cite{EKMM}), a fact we use in the
following construction.  Let $f\colon C\to D$ be a map of operads over
$\Lin,$ so there is a forgetful functor
\[
    f^{*} \colon \alg{D}{\spectra} \rightarrow \alg{C}{\spectra}.
\]
We construct the left adjoint $f_{!}$ of $f^{*}$ as a certain
coequalizer in $C$-algebras; see \cite[\S II.6]{EKMM} for further
discussion of this construction.

Denote by $m\colon DD\to D$ the multiplication for $D$, and let $A$ be a
$C$-algebra with structure map $\mu\colon CA\to A$.  Define $f_{!}A$ to be
the coequalizer in the diagram of $D$-algebras
\begin{equation} \label{eq:11}
\xymatrix{ {DCA} \ar@<.3ex>[rr]^{D\mu} \ar@<-.3ex>[rr] \ar[dr]_{Df} &
& {DA} \ar[r] & {f_{!}A.}
\\
& {DDA} \ar[ur]_{m} }
\end{equation}

In fact, it's enough to construct $f_{!}A$ as the coequalizer in
spectra.  Then $D$, applied to the unit $A\to CA$, makes the diagram a
reflexive coequalizer of spectra, and so $f_{!}A$ has the structure of
a $D$-algebra, and as such is the $D$-algebra coequalizer
\cite[\S II.6.6]{EKMM}. By construction, we have the following proposition.

\begin{Proposition} \label{t-pr-D-tens-C}
The functor $f_{!}$ is a continuous left adjoint to $f^{*}$; moreover, for any
spectrum $V$, the natural map
\begin{equation} \label{eq:8}
     f_{!} CV \to DV
\end{equation}
is an isomorphism.
\end{Proposition}

\begin{Remark}\label{rem-3}
Some treatments write $C\otimes V$ for the free $C$-algebra
$CV$, and then $D\otimes_{C}A$ for $f_{!}A$. 
\end{Remark}

About this adjoint pair there is the following well-known result,
which follows from the fact that $f^*$ preserves fibrations and weak
equivalences.

\begin{Proposition}\label{t-pr-coeo-lms-spectra}
Let $f\colon C\to D$ be a map of operads over $\Lin$.  The pair
$(f_{!},f^{*})$ is a continuous Quillen pair.
\end{Proposition}

It is folklore that all $E_{\infty}$ operads over $\Lin$ give rise
to the same homotopy theory.   Over the years, various arguments have
been given to show this, starting with May's use of the bar
construction to model $f_!$ (see \cite[\S II.4.3]{EKMM} for the most
recent entry in this line).  We present a model-theoretic formulation
of this result (under mild hypotheses on the operads) in the remainder
of the subsection.

\begin{Proposition} \label{t-pr-C-spectra-equiv-D-spectra}
If $f\colon C\to D$ is a map of $\einfty$ or $\ainfty$ operads, then $(f_{!},f^{*})$ is a
Quillen equivalence.  More generally, if each map
\[
    f\colon C (n) \to D (n)
\]
is a weak equivalence of spaces, then $(f_{!},f^{*})$ is a Quillen
equivalence.
\end{Proposition}

Before giving the proof, we make a few remarks.  Assume $f$ is a weak
equivalence of operads.  Since the pullback $f^* \colon \alg{D}{\spectra}
\to \alg{C}{\spectra}$ preserves fibrations and weak equivalences, to
show that $(f_!, f^*)$ is a Quillen equivalence it suffices to show
that for a cofibrant $C$-algebra $X$ the unit of the adjunction $X
\to f^* f_! X$ is a weak equivalence.

If $X = C Z$ is a \emph{free}
$C$-algebra, then $f_{!}X = f_{!} CZ \iso DZ$ by \eqref{eq:8}, and so
the map in question is the natural map
\[
      C Z \to DZ.
\]
It follows from Propositions X.4.7, X.4.9, and A.7.4 of \cite{EKMM} that
if the operad spaces $C (n)$ and $D (n)$ are CW-complexes, and if $Z$
is a wedge of spheres or disks, then $CZ\to DZ$ is a homotopy
equivalence.   In fact, this argument applies to the wider class of
\emph{tame} spectra, whose definition we now recall.

\begin{Definition}[\cite{EKMM}, Definition I.2.4]
A prespectrum $D$ is \emph{$\Sigma$-cofibrant} if each of the structure maps
$\Sigma^{W}D (V)\to D (V\oplus W)$ is a (Hurewicz) cofibration.  A
spectrum $Z$ is \emph{$\Sigma$-cofibrant} if it is isomorphic to one of the
form $LD$, where $D$ is a $\Sigma$-cofibrant prespectrum and $L$
denotes the spectrification functor~\cite[I.2.2]{LMS}.  A spectrum
$Z$ is \emph{tame} if it is homotopy equivalent to a
$\Sigma$-cofibrant spectrum.  In particular, a spectrum $Z$ of the
homotopy type of a CW-spectrum is tame.
\end{Definition}

For a general cofibrant $X$, the argument proceeds by reducing to
the free case $X=CZ$. In this paper, we present an inductive argument
due to Mandell \cite{Mandell:thesis}.  A different induction of this
sort appeared in \cite{Mandell:TAQ} in the algebraic setting; that
argument can be adapted to the topological context with minimal
modifications.

Our induction will involve the geometric realization
of simplicial spectra.  As usual, we would like to ensure that a map
of simplicial spectra
\[
f_{\bullet}\colon K_{\bullet} \to K'_{\bullet}
\]
in which each $f_{n}\colon K_{n} \to K'_{n}$ is a weak equivalence yields a
weak equivalence upon geometric realization.   The required condition
is that the spectra $K_{n}$ and $K_{n}'$ are tame: Theorem X.2.4 of
\cite{EKMM} says that the realization of weak equivalences of tame
spectra is a weak equivalence if $K_{\bullet}$ and $K'_{\bullet}$ are
``proper'' \cite[\S X.2.1]{EKMM}.  Recall that a simplicial spectrum
$K_\bullet$ is proper if the natural map of coends
\[
\int^{D_{q-1}} K_p \Smash D(q,p)_+ \to \int^{D_q} K_p \Smash D(q,p)_+
\cong K_q
\]
is a Hurewicz cofibration, where $D$ is the subcategory of $\Delta$
consisting of the monotonic surjections (i.e. the degeneracies), and
$D_q$ is the full subcategory of $D$ on the objects $0 \leq i \leq
q$.  This is a precise formulation of the intuitive notion that the
inclusion of the union of the degenerate spectra $s_j K_{q-1}$ in
$K_q$ should be a Hurewicz cofibration.

Thus, to ensure that the spectra that arise in our argument are tame
and the simplicial objects proper, we make the following simplifying
assumptions on our operads.

\begin{enumerate}
\item We assume that the spaces $C(n)$ and $D(n)$ have the
  homotopy type of $\Sigma_{n}$-$CW$-complexes.
\item We assume that $C(1)$ and $D(1)$ are equipped with
  nondegenerate basepoints.
\end{enumerate}

We believe these assumptions are reasonable, insofar as they are
satisfied by many natural examples; for instance, the linear
isometries operad and the little $n$-cubes operad both satisfy the
hypotheses above (see \cite[XI.1.4, XI.1.7]{EKMM} and
\cite[4.8]{May:gils} respectively).  More generally, if $\mathcal{O}$
is an arbitrary operad over the linear isometries operad, then taking
the geometric realization of the singular complex of the spaces
$\mathcal{O}$ produces an operad $|S (\mathcal{O})|$ with the
properties we require.

Goerss and Hopkins have proved two versions of Proposition
\ref{t-pr-C-spectra-equiv-D-spectra} using resolution model structures
to resolve an arbitrary cofibrant $C$-space by a simplicial $C$-space
with free $k$-simplices for every $k$.  A first version
\cite{GH:mpsht} proves the Proposition for Lewis-May-Steinberger
spectra, avoiding our 
simplifying assumptions on the operads via a detailed study of
``flatness'' for spectra (as an alternative to the theory of
``tameness'').  A more modern treatment \cite{GH:obstruct} works with
operads of simplicial sets and symmetric spectra in topological
spaces.  In that case, as they explain, a key point is that if $X$ is
a cofibrant spectrum, then $X^{(n)}$ is a \emph{free}
$\Sigma_{n}$-spectrum (see Lemma 15.5 of \cite{MR1806878}).  This
observation helps explain why the general form of the Proposition is
reasonable, even though the analogous statement for spaces is much too
strong.  We now give the proof of
Proposition~\ref{t-pr-C-spectra-equiv-D-spectra} under the
hypotheses enumerated above.

\begin{proof}
A cofibrant $C$-spectrum is a retract of a cell $C$-spectrum, and
so we can assume without loss of generality that $X$ is a cell
$C$-spectrum.  The argument for
Proposition~\ref{t-pr-C-spectra-level-model} implies that cell objects 
can be described as $X = \colim_n X_n$, where $X_0 = C(*)$ and
$X_{n+1}$ is obtained from $X_n$ via a pushout (in $C$-algebras) of
the form
\[
\xymatrix{
C A \ar[r] \ar[d] & X_n \ar[d] \\
C B \ar[r] & X_{n+1} \\
}
\]
where $A \to B$ is a wedge of generating cofibrations of spectra.
Furthermore, by the proof of Proposition~\ref{t-pr-C-spectra-level-model}
(specifically, the Cofibration Hypothesis), the map $X_n \to X_{n+1}$
is a Hurewicz
cofibration of spectra.   The hypotheses on $C$ and the fact that $A$
and $B$ are CW-spectra imply that $CA$ and $CB$ have the homotopy
type of CW-spectra, and thus inductively so does $X_n$.  Therefore,
since $f_!$ is a left adjoint, it suffices to show that $X_n \to f^*
f_! X_n$ is a weak equivalence for each $X_n$ --- under these
circumstances, a sequential colimit of weak equivalences is a weak
equivalence.

We proceed by induction on the number of stages required to build the
$C$-spectrum.  The base case follows from the remarks preceding the
proof.  For the induction hypothesis, assume that $f_!$ is a weak
equivalence for all cell $C$-algebras that can be built in $n$ or
fewer stages.  The spectrum $X_{n+1}$ is a pushout $C B \coprod_{C
  A} X_n$ in $C$-algebras, and this pushout is homeomorphic to a bar
construction $B(C B, C A, X_n)$, which is the geometric realization of
a simplicial spectrum where the $m$th is the coproduct $C B \coprod^m
C A \coprod X_n)$.  Since $f_!$ is a continuous left adjoint, it
commutes with geometric realization and coproducts in $C$-algebras,
and so $f_!(B(C B, C A, X_n))$ is homeomorphic to $B(D B, D A, f_!
X_n)$.

The bar constructions we are working with are proper simplicial
spectra by the hypothesis that $C(1)$ and $D(1)$ have
nondegenerate basepoints, and thus it suffices to show that at each
level in the bar construction
\[B_q (C B, C A, X_n) \to B_q (D B, D A, f_! X_n)\]
we have a weak equivalence of tame spectra.  This follows
from the inductive hypothesis: we have already shown that the spectra
are tame, and $C B \coprod^q C A \coprod X_n$ can be built in $n$
stages, since $X_n$ can be built in $n$ stages and the free algebras
can be built and added in a single stage.
\end{proof}

The idea of the following corollary goes all the way back to
\cite{May:gils}.

\begin{Corollary}\label{t-co-ho-einfty-well-defined}
If $C$ and $D$ are any two $\einfty$ operads over the linear
isometries operad, then the categories of $C$-algebras and
$D$-algebras are connected by a zig-zag of continuous Quillen
equivalences.
\end{Corollary}

\begin{proof}
Proposition \ref{t-pr-C-spectra-equiv-D-spectra} allow us to compare
each of categories of algebras to algebras over the linear isometries
operad.
\end{proof}

Backed by this result, we adopt the following convention.

\begin{Definition}\label{def-8}
We write $\ho \einftyspectra$ for the homotopy category of $\einfty$
ring spectra.  By this we mean the homotopy category $\ho
\alg{C}{\spectra}$ for any $\einfty$ operad $C$ over the linear
isometries operad.
\end{Definition}

\subsection{$\einfty$ spaces} \label{sec:einfty-spaces}

We adopt notation for operad actions on spaces analogous to our
notation for spectra in \S\ref{sec:einfty-spectra}.  Let $C$ be an
operad in topological spaces.  The free $C$-algebra on a space $X$ is
\begin{equation} \label{eq:25}
   CX = \coprod_{k\geq 0} C (k) \times_{\Sigma_{k}} X^{k}.
\end{equation}
We set $C (\emptyset) = C (0).$ The category of $C$-algebras in
spaces, or $C$-spaces, will be denoted $\alg{C}{\spaces}.$

Note that the sequence of spaces given by
\begin{align*}
    P (0) &  = \ptspace = P (1) \\
   P (k) & = \emptyset  \text{ for }k>1
\end{align*}
has a unique structure of operad, whose associated monad is
\[
    PX = \pt{X},
\]
so
\[
     \alg{P}{\spaces}\iso \ptdspaces.
\]

If $C$ is a unital operad and if $Y$ is a pointed space, let $\ptC Y$
be the pushout in the category of $C$-algebras
\begin{equation} \label{eq:42}
\xymatrix{
    C\ptspace \ar[r] \ar[d] & \ar[d] C (\emptyset) = \ptspace \\
CY \ar[r] & \ptC Y.
}
\end{equation}
Then $\ptC$ participates in a monad on the category of pointed spaces.
Indeed $\ptC$ is isomorphic to the monad $C_{May}$ introduced in
\cite{May:gils}, since for a test $C$-space $T$,
\[
    \alg{C}{\spaces}(\ptC Y,T)\iso \ptdspaces (Y,T) \iso
\alg{C}{\spaces} (C_{May}Y,T).
\]
There is a natural isomorphism
\[
     C X \iso \ptC (\pt{X}),
\]
and an equivalence of categories
\begin{equation} \label{eq:43}
        \alg{C}{\spaces} \iso \alg{\ptC}{\ptdspaces}.
\end{equation}
Part of this equivalence is the observation that, if $X$ is a
$C$-algebra, then it is a $\ptC$ algebra via
\[
   \ptC X \to \ptC (\pt{X}) \iso C X \to X.
\]
We have the following analogue of Proposition~\ref{t-pr-C-spectra-level-model}.

\begin{Proposition}\label{t-pr-C-spaces-level-model}
\hspace{5 pt}
\begin{enumerate}
\item \label{item:1} The category $\alg{C}{\spaces}$ has the structure of a
cofibrantly generated topological closed model category, in which the
forgetful functor to $\spaces$ creates fibrations and weak
equivalences.  If $\{A\to B \}$ is a set of generating (trivial)
cofibrations of $\spaces$, then $\{C A\to C B \}$ is a set of
generating (trivial) cofibrations of $\alg{C}{\spaces}$.
\item \label{item:2} The analogous statements hold for $\ptC$ and
$\alg{\ptC}{\ptdspaces}.$
\item Taking $C=P$, the resulting model category structure on the
category $\alg{P}{\spaces} \iso \ptdspaces$ is determined by the
forgetful functor to $\spaces$.
\item The equivalence $\alg{C}{\spaces}\iso \alg{\ptC}{\ptdspaces}$
\eqref{eq:43} carries the model structure arising from part
\eqref{item:1} to the model structure arising from part \eqref{item:2}.
\end{enumerate}
\end{Proposition}

\begin{proof}
The statements about the model structure on $\alg{C}{\spaces}$ or on
$\alg{\ptC}{\ptdspaces}$ can be proved for example by adapting the
argument in \cite{EKMM,MR1806878}.  The third part is standard, and
together the first three parts imply the last.
\end{proof}

We conclude this subsection with two results
which will be useful in \S\ref{sec:einfty-spac-conn}.
For the first, note that a point of $C (0)$ determines a map of
operads
\[
   P \to C,
\]
and so we have a forgetful functor
\[
  \alg{C}{\spaces} \to \alg{P}{\spaces} \iso \ptdspaces.
\]
We say that a point of $Y$ is \emph{non-degenerate} if $(Y,\ptspace)$ is
an NDR pair, i.e. that $\ptspace \to Y$ is a Hurewicz cofibration.

\begin{Proposition}\label{t-pr-cof-c-alg-cof-ptd-space}
Suppose that $C$ is a unital operad in topological spaces (or more
generally, an operad in which the base point of $C (0)$ is
nondegenerate).  If $X$ is a cofibrant object of
$\alg{\ptC}{\ptdspaces},$ then its base point is nondegenerate.
\end{Proposition}

Note that Rezk
\cite{Rezk:thesis} and Berger and Moerdijk \cite{MR2016697} have
proved a similar result, for algebras in a general model category over
an \emph{cofibrant} operad.  In our case, we need only assume that the
zero space $C (0)$ of our operad has a non-degenerate base point.

\begin{proof}
In the model structure described in Proposition
\ref{t-pr-C-spaces-level-model}, a cofibrant object is a retract of a
cell object, and so we can assume without loss of generality that $X$
is a cell $C$-space.  That is,
\begin{equation}\label{eq:50}
X = \colim_{n} X_{n}
\end{equation}
where $X_{0} = C (\emptyset)$ and $X_{n+1}$ is obtained from $X_n$ as
a pushout in $C$-spaces
\begin{equation} \label{eq:49}
\xymatrix{
C A \ar[r] \ar[d] & X_n \ar[d] \\
C B \ar[r] & X_{n+1}, \\
}
\end{equation}
where $A \to B$ is a disjoint union of generating cofibrations of
$\spaces$.

Our argument relies on a form of the Cofibration Hypothesis of \S{VII} of
\cite{EKMM}.  The key points are the following.

\begin{enumerate}
\item By assumption $X_0 = C (\emptyset) = C (0)$ is non-degenerately
based.
\item The space underlying the $C$-algebra colimit $X$ in
\eqref{eq:50} is just the space-level colimit.
\item In the pushout above,
\[
   X_{n} \to X_{n+1}
\]
is a based map and an unbased Hurewicz cofibration.
\end{enumerate}

The second point is easily checked (and is the space-level analog of
Lemma 3.10 of \cite{EKMM}).  For the last part, the argument in
Proposition 3.9 of \S{VII} of \cite{EKMM} (see also Lemma 15.9 of
\cite{MR1806878}) shows that the pushout \eqref{eq:49} is isomorphic
to a two-sided bar construction $B(C B, C A, X_n)$: this is the
geometric realization of a simplicial space where the $k$-simplices
are given as
$$C B \coprod_{C} (C A)^{\coprod k}\coprod_{C} X_n,$$ and the
simplicial structure maps are induced by the folding map and the maps
$C A \to C B$ and $C A \to X_n$.  Note that by $\coprod_{C}$ we mean
the coproduct in the category of $C$-spaces.  Recall that coproducts
(and more generally all colimits) in $C$-spaces admit a description as
certain coequalizers in $\spaces$.  Specifically, for $C$-spaces $X$
and $Y$ the coproduct $X \coprod_{C} Y$ can be described as the
coequalizer in $\spaces$
\[
\xymatrix{ C (C X \coprod C Y) \ar@<1ex>[r] \ar@<-1ex>[r] & C(X
\coprod
Y) \ar[r]&  X\coprod_{C} Y, \\
}
\]
where the unmarked coproducts are taken in $\spaces$ and the maps are
induced from the action maps and the monadic structure map,
respectively.  Following an argument along the lines of \cite[\S
VII.6]{EKMM}, we can show that for any $C$-algebra $A$ and space $B$,
the map $A \to A \coprod_{C} C B$ is an inclusion of a component in a
disjoint union. 

This implies that the simplicial degeneracy maps in the bar
construction are unbased Hurewicz cofibrations and hence that the
simplicial space is proper, that is, Reedy cofibrant in the
Hurewicz/Str\o{}m model structure.  Thus the inclusion of the zero
simplices $C B \coprod_{C} X_n$ in the realization is an unbased
Hurewicz cofibration, and hence the map $X_n \to X_{n+1}$ is itself a
unbased Hurewicz cofibration.  As a map of $C$-algebras, it's also a
based map.
\end{proof}

The second result we need is the following.

\begin{Proposition}\label{t-pr-cx-cw}
Let $C$ be an operad and suppose that each $C(n)$ has the homotopy
type of a $\Sigma_{n}$-$CW$ complex.  Let $X$ be a $C$-space with the
homotopy type of a cofibrant $C$-space.  Then $CX$ has the homotopy
type of a cofibrant $C$-space and the underlying space of $X$ has the
homotopy type of a $CW$-complex.
\end{Proposition}

\begin{proof}
The first statement is an easy consequence of the fact that $C$
preserves homotopies and cofibrant objects.  To see the second,
observe that the forgetful functor preserves homotopies, so it
suffices to suppose that $X$ is a cofibrant $C$-space.  Under our
hypotheses on $C$, if $A$ has the homotopy type of a CW-complex then
so does the underlying space of $CA$ (see for instance page 372 of
\cite{LMS} for a proof).  The result now follows from an
inductive argument along the lines of the preceding proposition.
\end{proof}

\subsection{$\einfty$ spaces and $\einfty$ spectra}

\label{sec:einfty-spaces-einfty-1}

Suppose that $C\to \Lin$ is an operad over $\Lin$.  In this section we
recall the proof of the following result:

\begin{Proposition}[\cite{MQRT:ersers}, \cite{LMS}
p. 366]\label{t-pr-l-spaces-l-spectra} The continuous Quillen pair
\begin{equation} \label{eq:2}
       \splus\colon \spaces \leftrightarrows \spectra\colon \linf
\end{equation}
induces by restriction a continuous Quillen adjunction
\begin{equation} \label{eq:3}
       \splus\colon \alg{\ptC}{\ptdspaces}\iso
\alg{C}{\spaces}\leftrightarrows \alg{C}{\spectra}\colon \linf
\end{equation}
between topological model categories.
\end{Proposition}

The first thing to observe is that $C$ and $\splus$ satisfy a strong
compatibility condition.

\begin{Lemma}\label{t-le-C-commutes-with-susp}
There is a natural isomorphism
\begin{equation}\label{eq:4}
    C \splus X \iso \splus C X.
\end{equation}
\end{Lemma}

\begin{proof}
It follows from \S{VI}, Proposition 1.5 of \cite{LMS} that, if
$X$ is a space, then
\[
   C (k) \ltimes_{\Sigma_{k}} (\splus X)^{\wedge k} \iso \splus (C (k)
\times_{\Sigma_{k}} X^{k}),
\]
and so
\begin{align*}
   C \splus X = \bigvee_{k\geq 0} C(k) \ltimes_{\Sigma_{k}} (\splus
X)^{\wedge k} \iso \bigvee_{k\geq 0} \splus (C (k)\times X^{k}) \\
\iso
\splus\left(\coprod_{k\geq 0} C (k) \times X^{k} \right) = \splus CX.
\end{align*}
\end{proof}

Next we have the following, from \cite[p. 366]{LMS}.

\begin{Lemma}  \label{t-pr-l-spaces-l-spectra-adj}
The adjoint pair
\begin{equation} \label{eq:23}
       \splus\colon \spaces \leftrightarrows \spectra\colon \linf
\end{equation}
induces an adjunction
\begin{equation} \label{eq:36}
       \splus\colon \alg{C}{\spaces}\leftrightarrows \alg{C}{\spectra}\colon
\linf
\end{equation}
and so also
\[
       \splus\colon \alg{\ptC}{\ptdspaces}\iso
\alg{C}{\spaces}\leftrightarrows \alg{C}{\spectra}\colon \linf
\]
\end{Lemma}

\begin{proof}
We show that the adjunction \eqref{eq:23} restricts to the
adjunction \eqref{eq:36}.  If $X$ is a $C$-space with structure map
$\mu\colon CX \to X$, then, using the isomorphism \eqref{eq:4}, $\splus X$
is a $C$-algebra via
\[
     C\splus X \iso \splus C X \xra{\splus \mu} \splus X.
\]
If $A$ is a $C$-spectrum, then $\linf A$ is a $C$-space via
\[
    C \linf A \rightarrow \linf C A \xra{} \linf A.
\]
The second map is just $\linf$ applied to the $C$-structure on $A$;
the first map is the adjoint of the map
\[
    \splus C \linf A \iso C \splus \linf A \rightarrow C A
\]
obtained using the counit of the adjunction.
\end{proof}

This adjunction allows us to prove the pointed analogue of Lemma
\ref{t-le-C-commutes-with-susp}.

\begin{Lemma}[\cite{LMS}, \S VII, Prop. 3.5] \label{t-le-ptC-adjunction}
If $C$ is a unital operad over $\Lin,$ then there is a natural isomorphism
\begin{equation} \label{eq:45}
    \splus\ptC Y \iso \ptC \splus Y \iso C\sinf Y.
\end{equation}
\end{Lemma}

\begin{proof}
Let $Y$ be a pointed space.  By Lemma
\ref{t-pr-l-spaces-l-spectra-adj} and the isomorphism \eqref{eq:4},
applying the left adjoint $\splus$ to the pushout diagram
\eqref{eq:42} defining $\ptC Y$ identifies $\splus \ptC Y$ with the
pushout of the diagram \eqref{eq:41} defining $\ptC \splus Y.$ The
second isomorphism is just the isomorphism \eqref{eq:44} together with
the isomorphism (for pointed spaces) $Y$
\[
    \splus Y \iso \sinf (S \vee Y).
\]
\end{proof}

\begin{proof}[Proof of Proposition \ref{t-pr-l-spaces-l-spectra}] It
remains to show that the adjoint pair $(\splus,\linf)$ induces a
Quillen adjunction.  For this it suffices to show that the right
adjoint $\linf$ preserves fibrations and weak equivalences (see, for
example, \cite[Lemma 1.3.4]{Hovey:MC}).  Now recall that the forgetful
functor $\alg{C}{\spectra}\to\spectra$ creates fibrations and weak
equivalences, and similarly for $\spaces$ \cite{EKMM,MR1806878}.  It
follows that the functor
\[
     \linf\colon \alg{C}{\spectra}\to \alg{C}{\spaces}
\]
preserves fibrations and weak equivalences, since
\[
     \linf\colon \spectra\rightarrow \spaces
\]
does.
\end{proof}

\begin{Remark}
Note that if $A$ is an $\einfty$ ring spectrum, then $\linf A$ is an
$\einfty$ space in two ways: one is described above, and arises from
the multiplication on $A$.  The other arises from the additive
structure of $A$, i.e. the fact that $\linf A$ is an infinite loop
space.  Together these two $\einfty$ structures give an $\einfty$ ring
space in the sense of \cite{MQRT:ersers} (see also \cite{May:rant}).
\end{Remark}

\subsection{$\einfty$ spaces and group-like $\einfty$ spaces}
\label{sec:einfty-spaces-group}

Suppose that $C$ is a unital $E_\infty$ operad, and let
$X$ be a $C$-algebra in spaces.  The structure maps
\begin{align*}
  \ptspace \to C (0) &\to X \\
  C (2)\times X \times X &\to X
\end{align*}
correspond to a family of $H$-space structures on $X$ and give to
$\pi_{0}X$ the structure of a monoid. 

\begin{Definition}  \label{def-1}
$X$ is said to be \emph{group-like} if $\pi_{0}X$ is a group.  We
write $\gplike{C}{\spaces}$ for the full subcategory of
$\alg{C}{\spaces}$ consisting of group-like $C$-spaces.
\end{Definition}

Note that if $f\colon X\to Y$ is a weak equivalence of $C$-spaces, then $X$
is group-like if and only $Y$ is.

\begin{Definition}\label{def-2}
We write $\Ho \gplike{C}{\spaces}$ for the image of
$\gplike{C}{\spaces}$ in $\Ho \alg{C}{\spaces}$.  It is the full
subcategory of homotopy types represented by group-like spaces.
\end{Definition}

If $X$ is a $C$-space, notice that $\GL{X}$ defined as in
Definition~\ref{def:gl1} is a group-like $C$-space.

\begin{Proposition}  \label{t-pr-GL}
The functor $\GLsym$ is the right adjoint of the inclusion
\[
        \gplike{C}{\spaces} \rightarrow \alg{C}{\spaces}
\]
\end{Proposition}

\begin{proof}
If $X$ is a group-like $C$-space, and $Y$ is a $C$-space, then
\[
\alg{C}{\spaces} (X,Y) \iso \gplike{C}{\spaces} (X,\GL{Y});
\]
just as, if $G$ is a group and $M$ is a monoid, then
\[
    \CatOf{monoids} (G,M) = \CatOf{groups} (G,\GL{M}).
\]
\end{proof}

\subsection{Group-like $\einfty$ spaces and connective spectra}
\label{sec:einfty-spac-conn}

A guiding result of infinite loop space theory is that group-like
$\einfty$ spaces provide a model for connective spectra.  We take a
few pages to show how the primary sources (in particular
\cite{MR0420609,May:gils,MR0339152}) may be used to prove a
formulation of this result in the language of model categories.

To begin, suppose that $C$ is a \emph{unital} $\einfty$ operad, and
$f$ is a map of monads (on pointed spaces)
\[
    f\colon \ptC \to Q \eqdef \linf \sinf.
\]
For example, we can take $C$ to be a unital $\einfty$ operad over the
infinite little cubes operad, but it is interesting to note that any
map of monads will do. If $V$ is a spectrum, then $\linf V$ is
a group-like $C$-algebra, via the map
\[
     \ptC \linf V \xra{f} \linf \sinf \linf V \to \linf V.
\]
Thus we have a factorization
\begin{equation} \label{eq:21}
\xymatrix{ {\spectra} \ar[r]^-{\Omega^{f}} \ar[dr]_-{\linf} &
{\gplike{C}{\spaces}} \ar[d]
\\
& {\ptdspaces} }
\end{equation}
We next show that the functor $\Omega^{f}$ has a left adjoint
$\Sigma^{f}.$ By regarding a $C$-space $X$ as a pointed space via
$\ptspace \to C (0) \to X$, we may form the spectrum $\sinf X$.  Let
$\Sigma^{f} X$ be the coequalizer in the diagram of spectra
\[
\xymatrix{ {\sinf \ptC X} \ar@<.3ex>[rr]^{\sinf \mu} \ar@<-.3ex>[rr]
\ar[dr]_{\sinf f} & & {\sinf X} \ar[r] & {\Sigma^{f} X.}
\\
& {\sinf \linf \sinf X} \ar[ur] }
\]
Then we have the following.

\begin{Lemma} \label{t-le-Sigma-f-C-sinf}
The pair
\begin{equation} \label{eq:27}
   \Sigma^{f}\colon \alg{C}{\spaces} \leftrightarrows \spectra\colon \Omega^{f}
\end{equation}
is a Quillen pair.  Moreover, the natural transformation
\[
     \Sigma^{f}\ptC \to \sinf
\]
is an isomorphism.
\end{Lemma}

\begin{proof}
As mentioned in the proof of Proposition \ref{t-pr-D-tens-C}, it is
essentially a formal consequence of the construction that $\Sigma^{f}$
is the left adjoint of $\Omega^{f}$.  Given the adjunction, we find
that $\Sigma^{f}\ptC\iso \sinf$, since, for any pointed space $X$ and
any spectrum $V$, we have
\begin{align*}
\spectra (\Sigma^{f}\ptC X,V) & \iso
   \alg{C}{\spaces} (\ptC X,\Omega^{f}V) \\
  & \iso  \ptdspaces (X,\linf V) \\
  & \iso \spectra (\sinf X, V).
\end{align*}
To show that we have a Quillen pair, it suffices (\cite[Lemma
  1.3.4]{Hovey:MC}) to show that $\Omega^{f}$ preserves weak
equivalences and fibrations.  This follows from the commutativity of
the diagram \eqref{eq:21}, the fact that $\linf$ preserves weak
equivalences and fibrations, and the fact that the forgetful functor
\[
\alg{C}{\spaces} \rightarrow \spaces
\]
creates fibrations and weak equivalences.
\end{proof}

Lemma \ref{t-le-Sigma-f-C-sinf} implies that the pair
$(\Sigma^{f},\Omega^{f})$ induce a continuous Quillen adjunction
\[
   \Sigma^{f}\colon \alg{C}{\spaces} \leftrightarrows \spectra\colon
\Omega^{f}.
\]
It is easy to see that this cannot be a Quillen equivalence.  Instead,
one expects that it induces an equivalence between the homotopy
categories of \emph{group-like} $C$-spaces and \emph{connective}
spectra.  In~\cite[0.10]{MR1806878}, this situation is called a ``connective
Quillen equivalence.''  The rest of this subsection is devoted to the
proof of the following result along these lines:

\begin{Theorem}  \label{t-th-einfty-spaces-spectra}
Suppose that $C$ is a unital $E_\infty$ operad, equipped with a map of monads
\[
      f\colon C \to \linf \sinf.
\]
Suppose moreover that
\begin{enumerate}
\item the base point $\ptspace \to C (1)$ is non-degenerate, and
\item for each $n$, the $n$-space $C (n)$ has the homotopy type of a $\Sigma_{n}$-$CW$-complex.
\end{enumerate}
Then the adjunction $(\Sigma^{f},\Omega^{f})$
induces an equivalence of categories
\[
\xymatrix{ {\Sigma^{f}\colon \ho \gplike{C}{\spaces}} \ar@<.3ex>[r] & {\ho
\CatOf{connective spectra}\colon \Omega^{f}} \ar@<.3ex>[l] }
\]
enriched over $\ho \spaces.$
\end{Theorem}

\begin{Remark}
As observed in \cite[A.8]{May:gils}, adding a whisker to a degenerate
basepoint produces a new operad $C'$ from $C.$   Also if $C$ is a
unital $\einfty$ operad equipped with a map of monads $f\colon C\to \linf
\sinf$, then taking the geometric realization of the
singular complex of the spaces $C (n)$ produces
an operad $|S (C)|$ with the properties we require.
\end{Remark}

The following Lemma, easily checked, is implicit in \cite{MR1806878}.
Let
\[
   F\colon \mathcal{M} \leftrightarrows \mathcal{M'}\colon G
\]
be a Quillen adjunction between topological closed model categories.
Let $\mathcal{C} \subseteq \mathcal{M}$ and $\mathcal{C'} \subseteq
\mathcal{M'}$ be full subcategories, stable under weak equivalence, so
we have sensible subcategories $\ho \mathcal{C} \subseteq \ho
\mathcal{M}$ and $\ho \mathcal{C'} \subseteq \ho \mathcal{M'}.$
Suppose that $F$ takes $\mathcal{C}$ to $\mathcal{C'}$, and $G$ takes
$\mathcal{C'}$ to $\mathcal{C}.$

\begin{Lemma}\label{t-le-conn-equiv}
If, for every cofibrant $X\in \mathcal{C}$ and every fibrant $Y\in
\mathcal{C'}$, a map
\[
     \phi\colon FX \to Y
\]
is a weak equivalence if and only if its adjoint
\[
     \psi\colon X \to GY
\]
is, then $F$ and $G$ induce equivalences
\[
    F\colon \ho \mathcal{C} \leftrightarrows \ho \mathcal{C'}\colon G
\]
of categories enriched over $\ho \spaces.$ 
\end{Lemma}

The key result in our setting is the following classical proposition;
we recall the argument from \cite{May:gils,MR0339152}.

\begin{Proposition}\label{t-pr-unit-we-May}
Let $C$ be a unital $\einfty$ operad, equipped with a map of monads
\[
   f\colon C \to \linf \sinf.
\]
Suppose that the basepoint $\ptspace \to C (1)$ is non-degenerate, and
that each $C (n)$ has the homotopy type of a
$\Sigma_{n}$-$CW$-complex.
If $X$ is a cofibrant $C$-space, then the unit
of the adjunction
\[
    X \to \Omega^{f}\Sigma^{f} X
\]
is group completion, and so a weak equivalence if $X$ is group-like.
\end{Proposition}

The proof of the proposition follows from analysis of the following
commutative diagram of simplicial $C$-spaces:

\begin{equation} \label{eq:14}
\xymatrix{
 B_{\bullet} (\ptC ,\ptC ,X) \ar[r] \ar[d] & \ar[d] \Omega^{f}
 \Sigma^{f} B_{\bullet} (\ptC ,\ptC ,X) \\
 X \ar[r] & \Omega^{f} \Sigma^{f}X.
}
\end{equation}

Specifically, we will show that under the hypotheses, on passage to
realization the vertical maps are weak equivalences and the top
horizontal map is group completion.

We begin by studying the left-hand vertical map; the usual simplicial
contraction argument shows the underlying map of spaces is a homotopy
equivalence, and so on passage to realizations we have a weak
equivalence of $C$-spaces.

\begin{Lemma}\label{t-le-BCCX-to-X}
For any operad $C$ and any $C$-space $X$, the left vertical arrow is a
map of simplicial $C$-spaces and a homotopy equivalence of simplicial
spaces, and so induces a weak
equivalence of $C$-spaces
\[
  B (\ptC,\ptC,X) \to X
\]
upon geometric realization.
\end{Lemma}

The right vertical map is more difficult to analyze, because we do not
know that $\Sigma^{f}$ preserves homotopy equivalences of spaces.  May
\cite[12.3]{May:gils} shows that, for suitable simplicial pointed spaces
$Y_{\bullet}$, the natural map
\begin{equation} \label{eq:73}
         |\Omega Y_{\bullet}| \to \Omega|Y_{\bullet}|
\end{equation}
is a weak equivalence, and he explains in \cite[\S8]{May:rant} how
this weak equivalence gives rise to a weak equivalence of $C$-spaces
\[
        |\Omega^{f} \Sigma^{f} B_{\bullet} (\ptC,\ptC,X)|\to
	\Omega^{f}|\Sigma^{f}B_{\bullet} (\ptC,\ptC,X)| \iso
	\Omega^{f}\Sigma^{f}B (\ptC,\ptC,X)
\]
by passage to colimits.   (We note that in \cite{May:rant}, May
describes proving that \eqref{eq:73} is a weak equivalence as the
hardest thing in \cite{May:gils}.)  Therefore, to show that the map
\[
|\Omega^{f} \Sigma^{f} B_{\bullet} (\ptC,\ptC,X)| \to
 \Omega^{f}\Sigma^{f} X
\]
is a weak equivalence, it suffices to show that for cofibrant $X$, the
map 
\[
\Sigma^f B(\ptC, \ptC, X) \to \Sigma^{f} X
\]
is a weak equivalence.
As it is straightforward to check from the definition that
$\Sigma^{f}$ does preserve weak equivalences between $C$-spaces with
the homotopy type of cofibrant $C$-spaces, the desired result will
follow once we show that $B(\ptC, \ptC, X)$ has the homotopy type of
a cofibrant $C$-space if $X$ is cofibrant.

\begin{Lemma} \label{t-le-BCC-cof}
Suppose that $C$ is a unital operad, such that the base point $* \to
C(1)$ is non-degenerate and each $C(n)$ has the homotopy type of a
$\Sigma_{n}$-$CW$-complex.  Let $X$ be a cofibrant $C$-space.  Then $B
(\ptC,\ptC,X)$ has the homotopy type of a cofibrant $C$-space.
\end{Lemma}

\begin{proof}
With our hypotheses, it follows from Proposition \ref{t-pr-cx-cw} that
the spaces $\ptC^{n}X$ have the homotopy type of cofibrant
$C$-spaces.  By Proposition \ref{t-pr-cof-c-alg-cof-ptd-space}, the
simplicial space $B_{\bullet} (\ptC,\ptC,X)$ is proper.  Finally, we
apply an argument analogous to that of Theorem X.2.7 of \cite{EKMM} to
show that if $Y_{\bullet}$ is a proper $C$-space in which each level
has the homotopy type of a cofibrant $C$-space, then $|Y_{\bullet}|$
has the the homotopy type of a cofibrant $C$-space.
\end{proof}





Finally, we consider the top horizontal map in \eqref{eq:14}.   We have
\emph{isomorphisms} of simplicial $C$-spaces
\[
      \Omega^{f}\Sigma^{f}B_{\bullet} (\ptC,\ptC,X) \iso B_{\bullet}
(\Omega^{f}\Sigma^{f}\ptC,\ptC,X) \iso B_{\bullet}
(\Omega^{f}\sinf,\ptC,X) \iso B_{\bullet} (Q,\ptC,X)
\]
(we used the isomorphism $\Sigma^{f}\ptC\iso \sinf$ of Lemma
\ref{t-le-Sigma-f-C-sinf}), and so an isomorphism of $C$-spaces
\[
      B (Q,\ptC,X) \iso |\Omega^{f}\Sigma^{f} B_{\bullet} (\ptC ,\ptC ,X)|
\]
We then apply the following result from \cite[\S 2]{MR0339152}.

\begin{Lemma} \label{t-le-gp-completion}
Let $C$ be a unital $\einfty$ operad, equipped with a
map of monads
\[
  f\colon   \ptC  \to \linf \sinf.
\]
Let $X$ be a $C$-space (and so pointed via $C (0)\to X$).   Suppose that
the base point of $C (1)$ and the base point of $X$ are
non-degenerate.  Then the map
\[
    B (\ptC,\ptC,X) \to  B (Q,\ptC,X),
\]
and so
\[
   B (\ptC,\ptC,X) \to |\Omega^{f}\Sigma^{f} B_{\bullet} (\ptC ,\ptC ,X)|,
\]
is group-completion.
\end{Lemma}

\begin{proof}
The point is that in general
\[
    \ptC Y \to \linf \sinf Y
\]
is group-completion \cite{MR0339176,MR0436146,MR0402733}, and so we
have the level-wise group completion
\[
     \ptC (\ptC )^{n}X \to \linf \sinf (\ptC)^{n}X
\]
(see~\cite[2.2]{MR0339152}).

The argument requires the simplicial spaces involved to be ``proper,''
that is, Reedy cofibrant with respect to the Hurewicz/Str\o{}m model
structure on topological spaces, so that the homology spectral
sequences have the expected $E_{2}$-term.  May proves that they are,
provided that $(C (1),\ptspace)$ and $(X,\ptspace)$ are NDR-pairs.
\end{proof}

We can now finish the proof of Theorem~\ref{t-th-einfty-spaces-spectra}.

\begin{proof}
It remains to show that if $X$ is a group-like cofibrant $C$-algebra
and $V$ is a (fibrant) $(-1)$-connected spectrum, then a map
\[
      \phi\colon \Sigma^{f} X \to V
\]
is a weak equivalence if and only if its adjoint
\[
       \psi\colon X \to \Omega^{f} V
\]
is.  These two maps are related by the factorization
\[
    \psi\colon X \to \Omega^{f} \Sigma^{f} X \xra{\Omega^{f}\phi}
\Omega^{f} V.
\]
The unit of adjunction is a weak equivalence by Proposition
\ref{t-pr-unit-we-May}.  It follows that $\psi$ is a weak
equivalence if and only if $\Omega^{f}\phi$ is.  Certainly if $\phi$
is a weak equivalence, then so is $\Omega^{f}\phi.$ Since both
$\Sigma^{f}X$ and $V$ are $(-1)$-connected, if $\Omega^{f}\phi$ is a
weak equivalence, then so is $\phi.$
\end{proof}

\begin{Remark}
There is another perspective on
Theorem~\ref{t-th-einfty-spaces-spectra} which elucidates the role of
the ``group-like'' condition on $C$-spaces.  Define a map
\[
    \alpha\colon X\to Y
\]
of $C$-spaces to be a \emph{stable} equivalence if the induced map
\[
    \Sigma^{f}\alpha' \colon \Sigma^{f} X' \to \Sigma^{f}Y'
\]
is a weak equivalence, where $X'$ and $Y'$ are cofibrant replacements
of $X$ and $Y$.  The ``stable'' model structure on $C$-spaces is the
localization of the model structure we have been considering in which
the weak equivalences are the stable equivalences, and the
cofibrations are as before.

In this stable model structure a $C$-space is fibrant if and only if
it is group-like; compare the model structure on $\Gamma$-spaces
discussed in \cite{MR1670249} and~\cite[\S 18]{MR1806878}.  The
homotopy category associated with the stable model structure is
exactly $\ho \gplike{C}{\spaces}$, and so this is a better encoding of
the homotopy theory of $\aC$-spaces.
\end{Remark}

\subsection{Proof of Theorem \ref{t-th-units}}
\label{sec:units:-proof-theorem}

Let $C$ be unital $\einfty$ operad, equipped with a map of operads
\[
         C \to \Lin,
\]
a map of monads on pointed spaces
\[
         f\colon \ptC \to \linf \sinf,
\]
and satisfying the hypotheses of Theorem
\ref{t-th-einfty-spaces-spectra}.  For example, we can take $C$ to be
\[
     C = | \SingTxt  (\mathcal{C} \times \Lin)|,
\]
the geometric realization of the singular complex on the product
operad $\mathcal{C}\times \Lin$, where $\mathcal{C}$ is infinite
little cubes operad of Boardman and Vogt~\cite{MR0420609}.

Then we have a sequence of continuous adjunctions (the left adjoints
are listed on top, and connective Quillen equivalence is indicated by
$\Qeq$).
\[
\xymatrix{ {\splus \linf\colon \CatOf{$(-1)$-connected spectra}}
\ar@<-.3ex>[r]_-{\Omega^{f}} & {\gplike{C}{\spaces}}
\ar@<-.3ex>[l]_-{\Sigma^{f},\Qeq} \ar@<.3ex>[r] & {\alg{C}{\spaces}}
\ar@<.3ex>[l]^-{\GLsym} \ar@<.3ex>[r]^-{\splus} & {\alg{C}{\spectra}\colon
\glsym} \ar@<.3ex>[l]^-{\linf} }
\]
By Proposition \ref{t-pr-coeo-lms-spectra}, $\alg{C}{\spectra}$ is a
model for the category of $\einfty$ spectra.  This completes the proof
of Theorem \ref{t-th-units}.

\def\cprime{$'$}

\end{document}